\documentclass[11pt]{article}
\usepackage{amsmath,amsthm}
\usepackage{amssymb,mathrsfs,pifont}
\usepackage{bm,mathtools}

\usepackage{tcolorbox}
\tcbset{colback=gray!20,colframe=gray!20,left=-4mm,right=2mm, width=.93\textwidth}

\usepackage{xcolor}
\usepackage{geometry}
\usepackage[colorlinks=true,
linkcolor=blue,citecolor=blue,
urlcolor=blue]{hyperref}

\makeatletter
\def\@seccntDot{.}
\def\@seccntformat#1{\csname the#1\endcsname\@seccntDot\hskip 0.5em}
\renewcommand\section{\@startsection{section}{1}{\z@}%
{18\p@ \@plus 6\p@ \@minus 3\p@}%
{9\p@ \@plus 6\p@ \@minus 3\p@}%
{\large\bfseries\boldmath}}
\renewcommand\subsection{\@startsection{subsection}{2}{\z@}%
{15\p@ \@plus 6\p@ \@minus 3\p@}%
{6\p@ \@plus 6\p@ \@minus 3\p@}%
{\itshape}}
\renewcommand\subsubsection{\@startsection{subsubsection}{3}{\z@}%
{12\p@ \@plus 6\p@ \@minus 3\p@}%
{\p@}%
{}}
\makeatother

\usepackage{microtype}

\theoremstyle{plain}
\newtheorem{theorem}{Theorem}[section]
\newtheorem{lemma}{Lemma}[section]
\newtheorem{problem}{Problem}[section]
\newtheorem{corollary}{Corollary}[section]
\newtheorem{proposition}{Proposition}[section]

\theoremstyle{definition}

\newtheorem{remark}{Remark}[section]

\numberwithin{equation}{section}
\allowdisplaybreaks
\DeclareMathOperator{\ex}{ex}
\DeclareMathOperator{\EX}{EX}
\DeclareMathOperator{\spex}{spex}
\DeclareMathOperator{\SPEX}{SPEX}

\title{Hypergraph Extensions of Spectral Tur\'an Theorem}

\author{Lele Liu\footnote{School of Mathematical Sciences, Anhui University, Hefei 230601, P.R. China. 
E-mail: \texttt{liu@ahu.edu.cn}.}~,~~ 
Zhenyu Ni\footnote{Department of Mathematics and Statistics, Hainan University, Haikou 570228, P.R. China. 
E-mail: \texttt{995264@hainanu.edu.cn}}~,~~
Jing Wang\footnote{College of Mathematics and Information Science, Henan Normal University, Xinxiang 453007, P.R. China. 
E-mail: \texttt{wj517062214@163.com}}~,~~
Liying Kang\footnote{Corresponding author. Department of Mathematics, Shanghai University,
Shanghai 200444, P.R. China. E-mail: \texttt{lykang@shu.edu.cn}.}~\footnote{Newtouch Center for Mathematics 
of Shanghai University, Shanghai 200444, P.R. China.}
}

\date{}

\begin{document}
\maketitle

\begin{abstract}
The spectral Tur\'an theorem states that the $k$-partite Tur\'an graph is the 
unique graph attaining the maximum adjacency spectral radius among all graphs 
of order $n$ containing no the complete graph $K_{k+1}$ as a subgraph. This 
result is known to be stronger than the classical Tur\'an theorem. 
In this paper, we consider hypergraph extensions of spectral Tur\'an theorem. 
For $k\geq r\geq 2$, let $H_{k+1}^{(r)}$ be the $r$-uniform hypergraph obtained from $K_{k+1}$ 
by enlarging each edge with a new set of $(r-2)$ vertices. 
Let $F_{k+1}^{(r)}$ be the $r$-uniform hypergraph with edges: $\{1,2,\ldots,r\} =: [r]$ and 
$E_{ij} \cup\{i,j\}$ over all pairs $\{i,j\}\in \binom{[k+1]}{2}\setminus\binom{[r]}{2}$, 
where $E_{ij}$ are pairwise disjoint $(r-2)$-sets disjoint from $[k+1]$. 
Generalizing the Tur\'an theorem to hypergraphs, Pikhurko 
[J. Combin. Theory Ser. B, 103 (2013) 220--225] and Mubayi and Pikhurko
[J. Combin. Theory Ser. B, 97 (2007) 669--678] respectively determined 
the exact Tur\'an number of $H_{k+1}^{(r)}$ and $F_{k+1}^{(r)}$, and  
characterized the corresponding extremal hypergraphs.

Our main results show that $T_r(n,k)$, the complete $k$-partite $r$-uniform 
hypergraph on $n$ vertices where no two parts differ by more than one in size, 
is the unique hypergraph having the maximum $p$-spectral radius among all 
$n$-vertex $H_{k+1}^{(r)}$-free (resp. $F_{k+1}^{(r)}$-free) $r$-uniform 
hypergraphs for sufficiently large $n$. These findings are obtained by 
establishing $p$-spectral version of the stability theorems. Our results 
offer $p$-spectral analogues of the results by Mubayi and Pikhurko, and connect 
both hypergraph Tur\'an theorem and hypergraph spectral Tur\'an theorem 
in a unified form via the $p$-spectral radius. In addition, by taking 
$p\to\infty$ in our findings, we obtain the corresponding edge extremal 
results, providing new proofs of Mubayi and Pikhurko’s results from spectral viewpoint.
\par\vspace{2mm}

\par\vspace{2mm}

\end{abstract}

\section{Introduction}
\label{sec1}

Consider an $r$-uniform hypergraph (or $r$-graph for brevity) $H$ and a family 
of $r$-graphs $\mathcal{F}$, we say $H$ is \emph{$\mathcal{F}$-free} if $H$ 
does not contain any member of $\mathcal{F}$ as a subgraph. The 
\emph{Tur\'an number} $\ex(n, \mathcal{F})$ is the maximum number of edges 
of an $\mathcal{F}$-free $r$-graph on $n$ vertices (if $\mathcal{F}$ is a 
single $r$-graph $F$, we write $\ex (n, F)$ instead of $\ex (n, \{F\})$). 
Let $\EX (n, \mathcal{F})$ denote the family of $n$-vertex $F$-free $r$-graphs with 
$\ex (n, \mathcal{F})$ edges. The \emph{Tur\'an density} of $\mathcal{F}$ is defined as 
\[
\pi(\mathcal{F}):= \lim_{n\to\infty} \frac{\ex(n, \mathcal{F})}{\binom{n}{r}}.
\]
The limit is known to exist, see Katona, Nemetz, and Simonovits \cite{Katona-Nemetz-Simonovits1964}.

Determining Tur\'an numbers of graphs and hypergraphs is one of the central 
problems in extremal combinatorics. The seminal contribution to this topic 
is the well-known Mantel's theorem \cite{Mantel1907}, established in 1907,
which says that every triangle free graph with $n$ vertices has at most 
$\lfloor n^2/4\rfloor$ edges. In 1941, Tur\'an \cite{Turan1941} extended
this result to $K_{t+1}$-free graphs, proving that the Tur\'an graph 
$T(n,t)$, a complete $t$-partite graph on $n$ vertices with the part 
sizes as equal as possible, is the unique $K_{t+1}$-free graph with the 
maximum number of edges. There are many extensions and generalizations on 
Tur\'an's result. The most celebrated extension attributes to the famous 
Erd\H{o}s-Stone-Simonovits Theorem \cite{Erdos-Stone1946,Erdos-Simonovits1966}, 
which states that 
\[
\ex (n, F) = \bigg(1 - \frac{1}{\chi(F) - 1} + o(1)\bigg) \frac{n^2}{2},
\]
where $\chi (F)$ is the chromatic number of $F$. Hence, for ordinary graphs 
the problem of determining $\ex (n, F)$ was asymptotically solved when $F$ is 
non-bipartite, while the case of bipartite graphs is still wide 
open (see the comprehensive survey \cite{Furedi-Simonovits2013}).
 
By contrast with the graph case, there is comparatively little understanding of 
the hypergraph Tur\'an number, even for very simple $r$-graphs. A central problem, 
originally posed by Tur\'an \cite{Turan1961}, is to determine $\pi(K_t^{(r)})$, 
where $t > r \geq 3$, and $K_t^{(r)}$ is the complete $r$-graph on $t$ vertices. 
This is a natural extension of determining the Tur\'an number of $K_t$ for graphs.  
To date, no case with $t > r \geq 3$ of this question has been solved, even asymptotically. 
A comprehensive survey of known bounds on $\ex (n, K_t^{(r)})$ was given by 
Sidorenko \cite{Sidorenko1995}, see also the earlier survey of de Caen \cite{Caen1994}. 
Although the Tur\'an problem for complete hypergraphs $K_t^{(r)}$ remains unresolved, 
there are specific hypergraphs for which the problem has been solved either asymptotically 
or exactly. We will explore these cases in detail in Subsection \ref{subsec1-4}.
\par\vspace{2mm}

The aim of this paper is to study spectral analogs of Tur\'an-type problems for 
hypergraphs, which is a natural extension of the graph case that has garnered 
interest in recent years.

\subsection{Spectral Tur\'an-type problem for graphs}

A significant class of extremal problems in spectral graph theory revolves around 
spectral analogues of classical Tur\'an-type problems. Although some initial findings 
existed, the systematic study of spectral Tur\'an-type problems was pioneered by 
Nikiforov \cite{Nikiforov2010}.

The study of spectral Tur\'an-type problems is interesting in its own right, it 
also has applications to studying classic Tur\'an-type problems. For example, 
Guiduli \cite{Guiduli1996} and Nikiforov \cite{Nikiforov2007} independently 
established a spectral analogue of the Tur\'an theorem \cite{Turan1941}, 
determining the maximum spectral radius of any $K_t$-free graph $G$ on $n$ vertices. 
By combining this result with the straightforward observation that the spectral 
radius provides an upper bound for the average degree of the graph, Guiduli \cite{Guiduli1996}
showed that the spectral Tur\'an theorem can imply the classical Tur\'an theorem.
Additionally, Nikiforov \cite{Nikiforov2010-2}, Babai and Guiduli \cite{Laszlo-Babai2009} 
established spectral versions of the K\H{o}v\'ari-S\'os-Tur\'an theorem for complete 
bipartite graphs $K_{s,t}$, yielding an upper bound that matches the best known upper 
bound for $\ex (n, K_{s,t})$ obtained by F\"uredi \cite{Furedi1996}. 

Recently, there has been an increase in publications focused on determining the maximum
spectral radius of graphs containing no copy of $F$ for various families of graphs $F$ 
have been published (see the surveys \cite{Nikiforov2011} and \cite{Li-Liu-Feng2022} for details) and 
this line of research has gained renewed interest. 
\par\vspace{3mm}

To state our results precisely, we need introduce the concept of $p$-spectral radius 
of hypergraphs in next subsection.

\subsection{$p$-spectral radius of hypergraphs}

For any real number $p\geq 1$, the $p$-spectral radius of hypergraphs was introduced 
by Keevash, Lenz and Mubayi \cite{Keevash-Lenz-Mubayi2014} and subsequently studied 
by Nikiforov \cite{Nikiforov2014}. Given an $r$-graph $H$ of order $n$, and a vector 
$\bm{x}=(x_1,x_2,\ldots,x_n)^{\mathrm{T}}\in\mathbb{R}^n$, denote
\[
P_H(\bm{x}) = r! \sum_{e\in E(H)} \prod_{v\in e} x_v.
\]
The \emph{$p$-spectral radius} of $H$ is defined as%
\begin{equation}\label{eq:definition-p-spectral-radius}
\lambda^{(p)}(H):= \max_{\|\bm{x}\|_p=1} P_H(\bm{x})
= \max_{\|\bm{x}\|_p=1} r! \sum_{e\in E(H)} \prod_{v\in e} x_v,
\end{equation}
where $\|\bm{x}\|_p:=(|x_1|^p+\cdots+|x_n|^p)^{1/p}$. If $\bm{x}\in\mathbb{R}^n$ is 
a vector with $\|\bm{x}\|_p=1$ such that $\lambda^{(p)}(H)=P_H(\bm{x})$, then $\bm{x}$ 
is called an \emph{eigenvector} corresponding to $\lambda^{(p)}(H)$.

For any real number $p\geq 1$, we denote by $\mathbb{S}_{p,+}^{n-1}$ the set of all 
nonnegative real vectors $\bm{x}\in\mathbb{R}^n$ with $\|\bm{x}\|_p=1$. If $\bm{x}\in\mathbb{R}^n$ 
is an eigenvector corresponding to $\lambda^{(p)}(H)$ with $\|\bm{x}\|_p = 1$, then 
the vector $\widetilde{\bm{x}} = (|x_1|, |x_2|, \ldots, |x_n|)^{\mathrm{T}}$ also 
satisfies $\|\widetilde{\bm{x}}\|_p = 1$ and 
\[
\lambda^{(p)} (H) = P_H(\bm{x}) \leq P_H(\widetilde{\bm{x}}) \leq \lambda^{(p)} (H),
\]
yielding that $\lambda^{(p)} (H) = P_H(\widetilde{\bm{x}})$. Hence, there is always 
a vector $\bm{x}\in\mathbb{S}_{p,+}^{n-1}$ such that $\lambda^{(p)} (H) = P_H(\bm{x})$.

By Lagrange's method, we have the eigenvalue\,--\,eigenvector equation for $\lambda^{(p)}(H)$ 
and $\bm{x}\in\mathbb{S}_{p,+}^{n-1}$ as follows:
\begin{equation}\label{eq:eigenequation}
\lambda^{(p)}(H)x_i^{p-1} = (r-1)! \sum_{\{i,i_2,\ldots,i_r\}\in E(H)}x_{i_2}\cdots x_{i_r}
~~\text{for}\ x_i>0.
\end{equation}

\begin{remark}
It is worth mentioning that the $p$-spectral radius $\lambda^{(p)}(H)$ shows remarkable 
connections with some hypergraph invariants. For instance, the quantity $\lambda^{(1)}(H)$ 
is the Lagrangian of $H$, the quantity $\lambda^{(2)}(H)$ is the notion of hypergraph 
spectral radius introduced by Friedman and Wigderson \cite{Friedman-Wigderson1995}, 
$\lambda^{(r)}(H)/(r-1)!$ is the usual spectral radius introduced by Cooper and Dutle \cite{Cooper2012}, 
$\lambda^{(\infty)}(H)/r!$ is the number of edges of $H$ (see \cite[Proposition 2.10]{Nikiforov2014}).
\end{remark}

\subsection{Spectral Tur\'an-type problem for hypergraphs}

In this paper, we focus our attention on the following problem which is a natural extension 
of graph case. 

\begin{problem}\label{problem:spectral-Turan}
Let $p > r-1 \geq 2$, and $\mathcal{F}$ be a family of $r$-graphs. What is the maximum $p$-spectral 
radius of an $r$-graph $H$ on $n$ vertices without no member of $\mathcal{F}$ as a subgraph?
\end{problem}

The motivation for investigating this problem, in addition to the reasons mentioned 
earlier for the graph case, stems from the opportunity to address both classical 
Tur\'an-type problems and spectral Tur\'an-type problems in a unified manner via the $p$-spectral radius. 

There are a few results on Problem \ref{problem:spectral-Turan} for $r>2$. 
Keevash, Lenz and Mubayi \cite{Keevash-Lenz-Mubayi2014} 
pioneered the study of the problem by introducing two general criteria that can be 
applied to obtain a variety of spectral Tur\'an-type results. In particular, they 
determined the maximum $p$-spectral radius of any $3$-graph on $n$ vertices not 
containing the Fano plane when $n$ is sufficiently large. Additionally, in the same 
paper they also obtained a $p$-spectral version of the Erd\H os-Ko-Rado theorem on 
$t$-intersecting $r$-graphs. Extending the spectral Mantel's theorem to hypergraphs, 
Ni, Liu and Kang \cite{Ni-Liu-Kang2024} determined the maximum $p$-spectral radius 
of $\{F_4, F_5\}$-free $3$-graphs, and characterized the extremal hypergraph, 
where $F_4 = \{abc, abd, bcd\}$ and $F_5 = \{abc, abd, cde\}$. 

Research has also been conducted on spectral Tur\'an-type problems over specific 
classes of hypergraphs. One such class is linear hypergraphs where any two 
distinct edges intersect at most one vertex. Gao, Chang and Hou \cite{GaoChangHou2022} 
studied the spectral extremal problem for $K_{r+1}^+$-free $r$-graphs among linear 
hypergraphs, where $K_{r+1}^+$ is the $r$-expansion of the complete graph $K_{r+1}$, 
i.e., $K_{r+1}^+$ is obtained from $K_{r+1}$ by enlarging each edge of $K_{r+1}$ with 
$(r - 2)$ new vertices disjoint from $V(K_{r+1})$ such that distinct edges of $K_{r+1}$ 
are enlarged by distinct vertices. They proved that the spectral radius of an $n$-vertices
$K_{r+1}^+$-free linear $r$-graph is no more than $n/r$ when $n$ is sufficiently large.
Generalizing Gao, Chang and Hou's result, She, Fan, Kang and Hou \cite{She-Fan-Kang-Hou2023} 
presented sharp (or asymptotic) bounds of the spectral radius of $F^+$-free linear 
$r$-graphs by establishing the connection between the spectral radius of linear 
hypergraphs and those of their shadow graphs, 
where $F$ is a graph with chromatic number $k$ with $k \geq r +1$. Wang and 
Yu \cite{Wang-Yu2023} obtain some upper bounds on the spectral radius of linear 
$3$-graphs without containing Berge-cycle as a subgraph. Zhou, Yuan and 
Wang \cite{Zhou-Yuan-Wang2024} obtain some bounds for the spectral 
radius of connected Berge-$K_{3,t}$-free linear $r$-graphs. Another recent result, 
due to Ellingham, Lu and Wang \cite{Ellingham-Lu-Wang2022}, showed that the 
$n$-vertex outerplanar $3$-graph of maximum spectral radius is the unique $3$-graph 
whose shadow graph is the join of an isolated vertex and the path $P_{n-1}$. 
Very recently, Cooper, Desai, and Sahay \cite{Cooper-Desai-Sahay2024} 
investigated the eigenvectors associated with the $p$-spectral radius of 
extremal hypergraphs for a general class of hypergraph spectral Tur\'an-type 
problems. Their findings revealed in a strong sense that these eigenvectors 
have close to equal weight on each vertex.

\subsection{Main results}
\label{subsec1-4}

Let $k\geq r\geq 2$. An $r$-graph $H$ is called \emph{$k$-partite} if its vertex set 
$V(H)$ can be partitioned into $k$ sets so that each edge contains at most one vertex 
from each set. An edge maximal $k$-partite $r$-graph is called \emph{complete $k$-partite}. 
Let $T_r(n,k)$ denote the complete $k$-partite $r$-graph on $n$ vertices where no 
two parts differ by more than one in size. We write $t_r(n,k)$ for the number of edges 
of $T_r(n,k)$. That is,
\begin{equation}\label{eq:size-T(n,k)}
t_r(n,k) = \sum_{S\in\binom{[k]}{r}} \prod_{i\in S} \bigg\lfloor \frac{n+i-1}{k}\bigg\rfloor 
= \big(1 - O(n^{-1})\big) \cdot \frac{(k)_r}{k^r} \binom{n}{r},
\end{equation}
where $[k]$ denotes the set $\{1,2,\ldots,k\}$, and $(k)_{r}$ 
denotes the falling factorial $k(k-1)\cdots (k-r+1)$.
 
Mubayi \cite{Mubayi2006} considered the Tur\'an problem for the following family of $r$-graphs.
For $t\geq r + 1$, let $\mathcal{K}_t^{(r)}$ be the family of all $r$-graphs $F$ with 
at most $\binom{t}{2}$ edges such that for some $t$-set $C$ (called the \emph{core}) every 
pair $u,v\in C$ is covered by an edge of $F$. When $r = 2$, the family $\mathcal{K}_t^{(r)}$ 
reduces to $K_t$, the usual complete graph, however when $r>2$, it contains more than one $r$-graph.
Let the $r$-graph $H_t^{(r)}\in \mathcal{K}_t^{(r)}$ be obtained from the complete $2$-graph 
$K_t$ by enlarging each edge with a new set of $(r-2)$ vertices. Mubayi \cite{Mubayi2006} 
showed that $\ex (n, \mathcal{K}_{k+1}^{(r)}) = t_r(n,k)$ with the unique extremal
$r$-graph being $T_r(n,k)$. Mubayi further established structural stability of near-extremal 
$\mathcal{K}_{k+1}^{(r)}$-free $r$-graphs. Using this stability property, Pikhurko \cite{Pikhurko2013} 
later strengthened Mubayi's result to show that $\ex (n, H_{k+1}^{(r)}) = t_r(n,k)$ for 
all sufficiently large $n$. For more results on Tur\'an number for $r$-graphs, we refer 
the reader to the surveys \cite{Keevash2011} and \cite{Mubayi-Verstraete2016}.

Our first main result provides a spectral analogue of Pikhurko's result.

\begin{theorem}\label{thm:main-1}
For any $p\geq r\geq 2$, there is $n_0$ such that for any $H_{k+1}^{(r)}$-free $r$-graph
$H$ on $n>n_0$ vertices, $\lambda^{(p)} (H) \leq \lambda^{(p)} (T_r(n,k))$,
with equality if and only if $H \cong T_r(n,k)$.
\end{theorem}

The cornerstone of our proof relies on a spectral version of stability result
(Theorem \ref{thm:spectral-stability-expanded-clique}), which says roughly that 
a near-extremal (with respect to $p$-spectral radius) $n$-vertex $H_{k+1}^{(r)}$-free 
$r$-graphs $H$ must be structurally close to $T_r(n,k)$. The stability method has 
been demonstrated to be a powerful tool for addressing Tur\'an-type problems.
It was first used by Simonovits in \cite{Simonovits1968} to determine $\ex (n,F)$
exactly for all color-critical graphs and large $n$, and then by several authors
to prove exact results for hypergraphs (see \cite{Keevash2011} for details). 
Theorem \ref{thm:spectral-stability-expanded-clique} provides a spectral 
counterpart of stability result for $H_{k+1}^{(r)}$. It seems that spectral 
version of stability theorems may play a valuable role in studying of spectral 
Tur\'an-type problems for hypergraphs.

\begin{remark}
If $r = 2$, Theorem \ref{thm:main-1} corresponds to the spectral Tur\'an theorem 
with respect to $p$-spectral radius, as established by Kang and Nikiforov \cite{Kang-Nikiforov2014} 
in 2014. 
\end{remark}

Since $H_{k+1}^{(r)}$ is a member of $\mathcal{K}_{k+1}^{(r)}$ and $T_r(n,k)$ is 
$\mathcal{K}_{k+1}^{(r)}$-free, applying Theorem \ref{thm:main-1} yields the following 
result directly.

\begin{corollary}
Let $p\geq r\geq 2$. Suppose that $H$ is a $\mathcal{K}_{k+1}^{(r)}$-free 
$r$-graph on $n$ vertices, then $\lambda^{(p)} (H) \leq \lambda^{(p)} (T_r(n,k))$ for 
sufficiently large $n$, with equality if and only if $H \cong T_r(n,k)$.
\end{corollary}

Let $G$ be a graph and $H$ be a hypergraph. The hypergraph $H$ is a \emph{Berge-$G$} 
if there is a bijection $\phi: E(G)\to E(H)$ such that $e\subseteq \phi(e)$ for all 
$e\in E(G)$. Alternatively, $H$ is a Berge-$G$ if we can embed each edge of $G$ into 
a unique edge of $H$.

\begin{corollary}
Let $p\geq r\geq 2$. Suppose that $H$ is a Berge-$K_{k+1}$-free $r$-graph
on $n$ vertices, then $\lambda^{(p)} (H) \leq \lambda^{(p)} (T_r(n,k))$ for 
sufficiently large $n$, with equality if and only if $H \cong T_r(n,k)$.
\end{corollary}

Note that for any fixed $r$-graph $H$, the $p$-spectral radius $\lambda^{(p)} (H)$
is a continuous function in $p$. Thus, by letting $p\to\infty$ in Theorem \ref{thm:main-1}, 
we immediately obtain Pikhurko's result.

\begin{corollary}[\cite{Pikhurko2013}]\label{coro:Pikhurko2013}
For any $k\geq r$, there is $n_0$ such that for any $H_{k+1}^{(r)}$-free $r$-graph
$H$ on $n>n_0$ vertices, $e(H) \leq t_r(n,k)$, with equality if and only if $H \cong T_r(n,k)$.
\end{corollary}

The Tur\'an number $\ex (n, Berge\text{-}K_{k+1})$ of Berge-clique can be derived 
from Theorem \ref{thm:main-1} by taking $p\to\infty$. This finding extends the 
corresponding results in \cite{Gerbner-Methuku-Palmer2020,Gyarfas2019,Gyori2006,Maherani-Shahsiah2018} 
when $n$ is sufficiently large.

\begin{corollary}
Let $H$ be a Berge-$K_{k+1}$-free $r$-graph on $n$ vertices. Then $e(H) \leq t_r(n,k)$ 
for sufficiently large $n$, with equality if and only if $H \cong T_r(n,k)$.
\end{corollary}

For $k\geq r\geq 2$, let $\mathcal{F}_k^{(r)}$ be the set of all minimal $r$-graphs 
$F$ such that there is a $k$-set $C$ (called the core) such that at least one edge 
$e\in E(F)$ lies entirely in $C$ and every pair of vertices of $C$ is covered by 
an edge of $F$. Let $F_k^{(r)}$ be the $r$-graph with edges: $[r]$ 
and $E_{ij} \cup\{i,j\}$ over all pairs $\{i,j\}\in \binom{[k]}{2}\setminus\binom{[r]}{2}$, 
where $E_{ij}$ are pairwise disjoint $(r-2)$-sets disjoint from $[k]$.

Mubayi and Pikhurko \cite{Mubayi-Pikhurko2007} proved that $\ex (n, F_{k+1}^{(r)}) = t_r(n,k)$ 
and $T_r(n,k)$ is the unique $r$-graph having $\ex (n, F_{k+1}^{(r)})$ edges for all sufficiently 
large $n$.

Our second main result contributes to Problem \ref{problem:spectral-Turan} by establishing 
a spectral counterpart of Mubayi-Pikhurko's result.

\begin{theorem}\label{thm:main-2}
For any $p\geq r\geq 2$, there is $n_0$ such that for any $F_{k+1}^{(r)}$-free $r$-graph
$H$ on $n>n_0$ vertices, $\lambda^{(p)} (H) \leq \lambda^{(p)} (T_r(n,k))$,
with equality if and only if $H \cong T_r(n,k)$.
\end{theorem}

Let Fan$^r$ denote the $r$-graph consisting of $r + 1$ edges $e_1,\ldots,e_r,e$,
with $e_i\cap e_j = \{u\}$ for all $i\neq j$, where $u\notin e$, and $|e_i\cap e|=1$ 
for all $i$. In other words, in Fan$^r$, $r$ edges share a single common vertex 
$u$ and the remaining edge intersects each of the other edges in a single vertex 
different from $u$. Thus, Fan$^r = F_{r+1}^{(r)}$. Clearly, Fan$^2$ is simply 
$K_3$, and in this sense Fan$^r$ generalizes the definition of 
triangle\footnote{Another way to generalize $K_3$, suggested by Katona \cite{Katona1974} 
and Bollob\'as \cite{Bollobas1974}, involves the concept of cancellative hypergraphs. 
An $r$-graph $H$ is call cancellative hypergraph if $H$ has no three distinct 
triples $A, B, C$ satisfying $B\triangle C \subset A$, where $\triangle$ denotes 
the symmetric difference. The spectral Tur\'an-type result for cancellative 
$3$-graphs was established in \cite{Ni-Liu-Kang2024}.}. 

\begin{corollary}
Let $p\geq r$. The unique $r$-graph attaining the maximum $p$-spectral radius among 
all $n$-vertex $r$-graphs containing no copy of Fan$^r$ is $T_r(n,r)$ for sufficiently large $n$.
\end{corollary}

\begin{corollary}[\cite{Mubayi-Pikhurko2007}]\label{coro:Mubayi-Pikhurko2007}
For any $k\geq r$, there is $n_0$ such that for any $F_{k+1}^{(r)}$-free $r$-graph
$H$ on $n>n_0$ vertices, $e(H) \leq t_r(n,k)$,
with equality if and only if $H \cong T_r(n,k)$.
\end{corollary}

This paper is organized as follows. In Section \ref{sec2} we introduce some
preliminary definitions and lemmas for the proofs of our main results.
Subsequently, in Section \ref{sec3} we establish spectral stability results 
for $H_{k+1}^{(r)}$ and $F_{k+1}^{(r)}$ respectively, which play a pivotal 
role in substantiating main results of this paper. The subsequent sections, 
Section \ref{sec4} and Section \ref{sec5}, are dedicated to the proofs of 
Theorem \ref{thm:main-1} and Theorem \ref{thm:main-2}, respectively.
We conclude this paper with some remarks and open problems in the last section.

\section{Preliminaries}
\label{sec2}

In this section we introduce definitions and notation that will be used throughout this paper, 
and give some preliminary results.

\subsection{Notation}

Consider an $r$-graph $H = (V(H), E(H))$ on $n$ vertices, where $E(H)$ is a collection 
of $r$-subsets of $V(H)$. For each vertex $v$, let $E_v$ denote the set of edges of $H$ 
containing $v$. The \emph{degree} $d_H(v)$ of $v$ is the size of $E_v$. 
The \emph{link} of $v$ in $H$ is defined as $\{e\setminus\{v\}: e\in E_v\}$. 
The \emph{codegree} of vertices $u$ and $v$ is the number of edges containing both $u$ and $v$. 
A \emph{$k$-independent set} in a hypergraph $H$ is a set $I \subseteq V(H)$ so that 
$|I\cap e| < k$ for all $e \in E(H)$. We refer to a $2$-independent set as a \emph{strong independent set}.
Let $H_1$ and $H_2$ be two $r$-graphs sharing the same vertex set $V$. Define $H_1\setminus H_2$ 
as the $r$-graph with vertex set $V$ and edge set $E(H_1)\setminus E(H_2)$. Also, we 
write $H_1 + H_2$ for the $r$-graph with vertex set $V$ and edge set $E(H_1) \cup E(H_2)$.

Let $H$ be an $r$-graph. The \emph{transversal number} $\tau (H)$ of $H$ is the minimum 
cardinality of a subset of $V(H)$ which intersects all edges of $H$. Given an $r$-graph 
$H$ of order $n$ and positive integers $k_1,\ldots,k_n$, write $H(k_1,\ldots,k_n)$ for 
the $r$-graph obtained by replacing each vertex $v\in V(H)$ with a set $U_v$ of size $k_v$ 
and each edge $\{v_1,\ldots,v_r\}\in E(H)$ with a complete $r$-partite $r$-graph with vertex 
classes $U_{v_1},\ldots,U_{v_r}$. The graph $H(k_1,\ldots,k_n)$ is called a \emph{blow-up} 
of $H$. For more definitions and notation from hypergraph theory, see, e.g. \cite{Bretto2013}.

\subsection{Bounds on the $p$-spectral radius}

We begin with the following inequality which may be regarded as a part of the well-known 
Weyl's inequalities for hypergraphs.

\begin{proposition}[Weyl's inequality, {\cite[Proposition 6.2]{Nikiforov2014}}]\label{prop:weyl-inequality}
Let $H_1$ and $H_2$ be two edge-disjoint $r$-graphs with same vertex set. Then 
\[
\lambda^{(p)} (H_1 + H_2) \leq \lambda^{(p)} (H_1) + \lambda^{(p)} (H_2).
\]
\end{proposition}

The subsequent lemma extends the evident inequality $\lambda^{(2)} (G) \leq \sqrt{2e(G)}$ 
from graphs to hypergraphs.

\begin{lemma}[{\cite[Lemma 1.8]{Keevash-Lenz-Mubayi2014}}]\label{lem:size-upper-bound}
Let $p>1$ and $H$ be an $r$-graph with $m$ edges. Then 
\[
\lambda^{(p)} (H) \leq (r! m)^{1-1/p}.
\]
\end{lemma}

\begin{lemma}[{\cite[Theorem 2]{Kang-Nikiforov-Yuan2015}}]\label{lem:k-partite-upper-bound}
Let $k\geq r\geq 2$, and let $H$ be a $k$-partite $r$-graph of order $n$. Then for each $p>1$,
\[
\lambda^{(p)} (H) < \lambda^{(p)} (T_r(n,k)),
\]
unless $H = T_r(n,k)$.
\end{lemma}

We also need the following lower bound on $\lambda^{(p)} (T_r(n,k))$. 

\begin{lemma}\label{lem:lambdaT}
Let $k\geq r\geq 2$. Then
\[ 
\lambda^{(p)} (T_r(n, k)) \geq \big( 1 - O(n^{-1}) \big) \cdot \frac{(k)_r}{k^r} n^{r(1-1/p)}.
\]
\end{lemma}

\begin{proof}
Let $\bm{x} = n^{-1/p} (1,1,\ldots, 1)^{\mathrm{T}} \in\mathbb{S}_{p,+}^{n-1}$.
By \eqref{eq:definition-p-spectral-radius} and \eqref{eq:size-T(n,k)}, we have 
\begin{align*}
\lambda^{(p)} (T_r(n, k)) 
& \geq \frac{r!\cdot e(T_r(n,k))}{n^{r/p}} \\
& = \big( 1 - O(n^{-1}) \big) \cdot \frac{(k)_r}{k^r} n^{r(1-1/p)},
\end{align*}
completing the proof of Lemma \ref{lem:lambdaT}.
\end{proof}



\subsection{Eigenvectors corresponding to $\lambda^{(p)}$}

 

The two lemmas that follow, while simple, are not obvious for arbitrary $p>1$.

\begin{lemma}[{\cite[Proposition 11]{Kang-Nikiforov-Yuan2015}}]\label{lem:complete-k-partite-x-positive}
If $H$ is a complete $k$-partite $r$-graph and $p > 1$, then every nonnegative
vector corresponding to $\lambda^{(p)} (H)$ is positive.    
\end{lemma}

\begin{lemma}[{\cite[Proposition 12]{Kang-Nikiforov-Yuan2015}}]\label{lem:subgraph-spectral-radius}
Let $p \geq 1$, and let $H$ be an $r$-graph such that every nonnegative vector
corresponding to $\lambda^{(p)} (H)$ is positive. If $G$ is a subgraph of $H$, then 
$\lambda^{(p)} (G) < \lambda^{(p)} (H)$, unless $G=H$.
\end{lemma}

We conclude this section with the following lemma, which can be proved by induction or double counting.

\begin{lemma}[{\cite[Lemma 12]{Cioaba-Feng-Tait-Zhang2020}}]\label{lem:intersection-sets}
Let $A_1,A_2,\ldots,A_k$ be $k$ finite sets. Then 
\[
|A_1\cap A_2\cap\cdots\cap A_k| \geq \sum_{i=1}^k |A_i| - (k-1) \bigg|\bigcup_{i=1}^k A_i\bigg|.
\]
\end{lemma}

\section{Hypergraph spectral stability results}
\label{sec3}

This section aims to demonstrate spectral stability results for $H_{k+1}^{(r)}$ 
and $F_{k+1}^{(r)}$ respectively, which are essential in substantiating our main results.

\begin{theorem}[Spectral Stability of $H_{k+1}^{(r)}$]\label{thm:spectral-stability-expanded-clique}
Let $p>1$ and $k\geq r\geq 3$. For any $\varepsilon > 0$, there are $\delta = \delta(k,r,\varepsilon) > 0$
and $n_0 = n_0(k,r,\varepsilon)$ such that the following holds for all $n>n_0$:
If $H$ is an $n$-vertex $H_{k+1}^{(r)}$-free $r$-graph with
\[
\lambda^{(p)} (H) > (1 - \delta) \frac{(k)_r}{k^r} n^{r(1-1/p)},
\]
then $H$ can be transformed to $T_r(n, k)$ by adding and deleting at most $\varepsilon\binom{n}{r}$ edge.
\end{theorem}

\begin{theorem}[Spectral Stability of $F_{k+1}^{(r)}$]\label{thm:spectral-stability-Fan}
Let $p>1$ and $k\geq r\geq 3$. For any $\varepsilon > 0$, there are $\delta = \delta(k,r,\varepsilon) > 0$
and $n_0 = n_0(k,r,\varepsilon)$ such that the following holds for all $n>n_0$:
If $H$ is an $n$-vertex $F_{k+1}^{(r)}$-free $r$-graph with
\[
\lambda^{(p)} (H) > (1 - \delta) \frac{(k)_r}{k^r} n^{r(1-1/p)},
\]
then $H$ can be transformed to $T_r(n, k)$ by adding and deleting at most $\varepsilon\binom{n}{r}$ edge.
\end{theorem}

In order to prove Theorem \ref{thm:spectral-stability-expanded-clique} and 
Theorem \ref{thm:spectral-stability-Fan}, we begin by revisiting relevant 
definitions and proving several preliminary lemmas. 

\subsection{Lagrangian density and spectral stability}

Let $k\geq r\geq 2$. 
We say that a family of $r$-graph $\mathcal{F}$ is \emph{stable} with respect to $T_r(n,k)$, if
for any $\varepsilon > 0$ there exists some positive $\delta$  
such that if $H$ is an $n$-vertex $\mathcal{F}$-free $r$-graph with at least 
$(1 - \delta) \frac{(k)_r}{k^r} \binom{n}{r}$ edges, then $H$ is 
$\varepsilon\binom{n}{r}$-close to $T_r(n,k)$, i.e., $H$ can be transformed 
to $T_r(n, k)$ by adding and deleting at most $\varepsilon\binom{n}{r}$ edges.

Likewise, we say that a family of $r$-graph $\mathcal{F}$ is \emph{spectrally stable} with respect 
to $T_r(n,k)$, if for any $\varepsilon > 0$ there exists $\delta > 0$ such that 
if $H$ is an $n$-vertex $\mathcal{F}$-free $r$-graph with  
$\lambda^{(p)} (H) > (1 - \delta) \frac{(k)_r}{k^r} n^{r(1-1/p)}$, then $H$ 
is $\varepsilon\binom{n}{r}$-close to $T_r(n,k)$.

Given a family $\mathcal{F}$ of $r$-graphs, the \emph{Lagrangian density} $\pi_{\lambda}(\mathcal{F})$ 
of $\mathcal{F}$ is defined as 
\[ 
\pi_{\lambda}(\mathcal{F}) := \sup\{\lambda^{(1)}(H): H\ \text{is}\ \mathcal{F} \text{-free}\}.
\]
The Lagrangian density is closely related to the Tur\'an density. It is not hard to see that 
$\pi (\mathcal{F}) \leq\pi_{\lambda} (\mathcal{F})$ (see, e.g. \cite{Jiang-Peng-Wu2018}). 
Given $r$-graphs $F$ and $H$ we say $f: V(F) \to V(H)$ is a \emph{homomorphism} if it 
preserves edges, i.e., $f(e)\in E(H)$ for all $e\in E(F)$. We say that $H$ is \emph{$F$-hom-free} 
if there is no homomorphism from $F$ to $H$. For a family $\mathcal{F}$ of $r$-graphs, 
$H$ is \emph{$\mathcal{F}$-hom-free} if it is $F$-hom-free for all $F\in\mathcal{F}$. 
Sidorenko \cite{Sidorenko1987} showed that $\pi (\mathcal{F})$ is equal to the supremum 
of the Lagrangians of all $\mathcal{F}$-hom-free $r$-graphs (see also \cite[Section 3]{Keevash2011}), i.e., 
\begin{equation}\label{eq:Hom-free-density}
\pi(\mathcal{F}) = \sup\{\lambda^{(1)} (H): H\ \text{is}\ \mathcal{F}\text{-hom-free}\}.
\end{equation}
 
We first establish a general criterion that guarantee an $r$-graph is spectrally 
stable with respect to $T_r(n,k)$. This criterion will play a key role in 
the proofs of Theorem \ref{thm:spectral-stability-expanded-clique} and 
Theorem \ref{thm:spectral-stability-Fan}.

\begin{lemma}\label{lem:general-criterion}
Let $k\geq r \geq 2$ and $p > 1$. Let $\mathcal{F}$ be a family of $r$-graph such that 
\begin{enumerate}
\item[$(1)$] $\mathcal{F}$ is stable with respect to $T_r(n,k)$.
\item[$(2)$] $\pi_{\lambda} (\mathcal{F}) \leq \frac{(k)_r}{k^r}$.
\end{enumerate}
Then $\mathcal{F}$ is spectrally stable with respect to $T_r(n,k)$.
\end{lemma}

\begin{proof}
Let $\varepsilon > 0$ be given. Since $\mathcal{F}$ is stable with respect to $T_r(n,k)$, 
there exists some positive $\delta_1$ such that any $n$-vertex $\mathcal{F}$-free 
$r$-graph with at least $(1 - \delta_1) \frac{(k)_r}{k^r} \binom{n}{r}$ edges 
is $\varepsilon\binom{n}{r}$-close to $T_r(n,k)$.

Let $H$ be an $n$-vertex $\mathcal{F}$-free $r$-graph, and $\bm{x}\in\mathbb{S}_{p,+}^{n-1}$ 
be an eigenvector corresponding to $\lambda^{(p)} (H)$. Let $\bm{y}$ be a vector of 
dimension $n$ with $y_u = x_u^p$ for $u\in V(G)$. By the definition of $p$-spectral 
radius and Power Mean inequality, we see
\begin{align*}
\lambda^{(p)} (H) 
& = r! \sum_{e\in E(H)} \prod_{u\in e} x_u \\
& = r! \sum_{e\in E(H)} \prod_{u\in e} y_u^{1/p} \\
& \leq r! (e(H))^{1-1/p} \left(\sum_{e\in E(H)} \prod_{u\in e} y_u\right)^{1/p}. 
\end{align*}
Taking into account that $\|\bm{y}\|_1 = 1$ and considering the definition of $\pi_{\lambda} (\mathcal{F})$, we arrive at 
\[
\sum_{e\in E(H)} \prod_{u\in e} y_u \leq \frac{\lambda^{(1)} (H)}{r!}
\leq \frac{\pi_{\lambda} (\mathcal{F})}{r!}.
\]
Combining these two inequalities, we immediately obtain that  
\begin{align}
\lambda^{(p)} (H) 
& \leq (r!\cdot e(H))^{1-1/p} \cdot \pi_{\lambda}(\mathcal{F})^{1/p} \label{eq:used-in-last-section} \\
& \leq (r!\cdot e(H))^{1-1/p} \bigg(\frac{(k)_r}{k^r}\bigg)^{1/p}. \label{eq:temporary-1}
\end{align}

On the other hand, choose $\delta = (1-1/p) \delta_1$ and assume that 
$\lambda^{(p)} (H) > (1 - \delta) \frac{(k)_r}{k^r} n^{r(1-1/p)}$. 
In light of \eqref{eq:temporary-1}, we deduce that
\[
(r! \cdot e(H))^{1-1/p} \cdot \bigg(\frac{(k)_r}{k^r}\bigg)^{1/p}
> (1 - \delta) \frac{(k)_r}{k^r} n^{r(1-1/p)},
\]
which, together with Bernoulli inequality, yields that 
\begin{align*}
e(H) & > (1-\delta)^{p/(p-1)} \frac{(k)_r}{k^r} \binom{n}{r} \\
& > \Big(1 - \frac{p\delta}{p-1} \Big) \frac{(k)_r}{k^r} \binom{n}{r} \\
& = (1 - \delta_1 ) \frac{(k)_r}{k^r} \binom{n}{r}.
\end{align*} 
Hence, $H$ is $\varepsilon \binom{n}{r}$-close to $T_r(n,k)$, 
completing the proof of Lemma \ref{lem:general-criterion}.
\end{proof}

\subsection{Spectral stability of $H_{k+1}^{(r)}$}

The main aim of this subsection is to prove Theorem \ref{thm:spectral-stability-expanded-clique}. 
To finish the proof, we first demonstrate the spectral stability result of $\mathcal{K}_{k+1}^{(r)}$ 
in light of Lemma \ref{lem:general-criterion}. Subsequently, we reduce an $H_{k+1}^{(r)}$-free 
$r$-graph on $n$ vertices to a $\mathcal{K}_{k+1}^{(r)}$-free $r$-graph by removing $O(n^{r-1})$ 
edges. During this process, we show that the $p$-spectral radius does not decrease much via some 
spectral inequalities. Finally, by combining these two steps, we achieve the spectral stability 
of $H_{k+1}^{(r)}$.

Given a vector $\bm{x}\in\mathbb{R}^n$, the \emph{support set} $S$ of $\bm{x}$ is the set of 
the indices of non-zero elements in $\bm{x}$, i.e., $S = \{i\in [n]: x_i\neq 0\}$.

\begin{lemma}[{\cite[Theorem 2.3]{Frankl-Rodl1984}}]\label{lem:Lagrangian-covered}
Let $H$ be an $r$-graph. Let $\bm{z}$ be an eigenvector corresponding to $\lambda^{(1)}(H)$ with 
$\|\bm{z}\|_1 = 1$ such that the support set $S$ of $\bm{z}$ is minimal. Then for 
each $u,v\in S$, there is an edge in $E(H)$ containing both $u$ and $v$. 
\end{lemma}

Using Lemma \ref{lem:Lagrangian-covered}, we obtain the following proposition. 

\begin{proposition}\label{prop:Lagrangian-density-1}
$\pi_{\lambda}(\mathcal{F}_{k+1}^{(r)}) \leq \frac{(k)_r}{k^r}$.
\end{proposition}

\begin{proof}
Let $H$ be an $\mathcal{F}_{k+1}^{(r)}$-free $r$-graph with at least one edge, and 
let $\bm{z}$ be a nonnegative eigenvector corresponding to $\lambda^{(1)}(H)$ with 
minimal suppose set $S$. We first show that $|S|\leq k$ by contradiction. If $|S| > k$, 
then there is no edge lies entirely in $S$, otherwise we will get a member of 
$\mathcal{F}_{k+1}^{(r)}$ by Lemma \ref{lem:Lagrangian-covered}. Hence,
\[
\lambda^{(1)} (H) = r!\sum_{e\in E(H)} \prod_{u\in e} z_u = 0,
\]
a contradiction with $\lambda^{(1)} (H) > 0$ as $e(H) \geq 1$.

Now, let $|S|\leq k$. Then
\begin{equation}\label{eq:temporary-2}
\lambda^{(1)} (H) = r!\sum_{e\in E(H)} \prod_{u\in e} z_u \leq r!\sum_{e\in E(K_k^{(r)})} \prod_{u\in e} z_u. 
\end{equation}
By Maclaurin's inequality, we obtain 
\[
\left(\frac{\sum_{e\in (K_k^{(r)})} \prod_{u\in e} z_u}{\binom{k}{r}}\right)^{1/r}
\leq \frac{\sum_{u\in S} z_u}{\binom{k}{1}} = \frac{1}{k},
\]
which, together with \eqref{eq:temporary-2}, gives that
$\lambda^{(1)} (H) \leq \frac{(k)_r}{k^r}$.
\end{proof}

A straightaway corollary of Proposition \ref{prop:Lagrangian-density-1} is 
$\pi_{\lambda}(\mathcal{K}_{k+1}^{(r)}) \leq \frac{(k)_r}{k^r}$ since 
$\mathcal{F}_{k+1}^{(r)} \subseteq \mathcal{K}_{k+1}^{(r)}$.

\begin{corollary}\label{coro:Lagrangian-density-1}
$\pi_{\lambda}(\mathcal{K}_{k+1}^{(r)}) \leq \frac{(k)_r}{k^r}$.
\end{corollary}

Mubayi \cite{Mubayi2006} proved that $\mathcal{K}_{k+1}^{(r)}$ is stable with
respect to $T_r(n,k)$.

\begin{lemma}[{\cite[Theorem 3]{Mubayi2006}}]\label{lem:stability-for-edge-1}
Let $k\geq r\geq 2$. For any $\varepsilon > 0$ there are $\delta = \delta(k,r,\varepsilon) > 0$
and $n_0 = n_0(k,r,\varepsilon)$ such that any $\mathcal{K}_{k+1}^{(r)}$-free $r$-graph $H$ of order 
$n\geq n_0$ and 
\[
e(H) > (1-\delta) \frac{(k)_r}{k^r} \binom{n}{r}
\]
is $\varepsilon\binom{n}{r}$-close to $T_r(n, k)$.
\end{lemma}

By Lemma \ref{lem:general-criterion}, Lemma \ref{lem:stability-for-edge-1} and 
Corollary \ref{coro:Lagrangian-density-1}, we obtain the spectral stability of 
$\mathcal{K}_{k+1}^{(r)}$ immediately.

\begin{lemma}[Spectral Stability of $\mathcal{K}_{k+1}^{(r)}$]\label{lem:spectral-stability-general-clique}
Let $p>1$ and $k\geq r\geq 2$. For any $\varepsilon > 0$, there are $\delta = \delta(k,r,\varepsilon) > 0$
and $n_0 = n_0(k,r,\varepsilon)$ such that the following holds for all $n>n_0$:
If $H$ is an $n$-vertex $\mathcal{K}_{k+1}^{(r)}$-free $r$-graph with
\[
\lambda^{(p)} (H) > (1 - \delta) \frac{(k)_r}{k^r} n^{r(1-1/p)},
\]
then $H$ is $\varepsilon \binom{n}{r}$-close to $T_r(n, k)$.
\end{lemma}

Now, we are ready to present the proof of Theorem \ref{thm:spectral-stability-expanded-clique}. 
\par\vspace{3mm}

\noindent\emph{Proof of Theorem \ref{thm:spectral-stability-expanded-clique}}.
Let $\varepsilon > 0$ be given. By Lemma \ref{lem:spectral-stability-general-clique}, there exist 
$\delta_1>0$ and $n_1$ such that any $\mathcal{K}_{k+1}^{(r)}$-free $r$-graph $G$ having $n\geq n_1$ 
vertices and $\lambda^{(p)} (G) > (1 - \delta_1) \frac{(k)_r}{k^r} n^{r(1-1/p)}$ is 
$\frac{\varepsilon}{2} \binom{n}{r}$-close to $T_r(n,k)$. 

Choose 
\[
\delta = \min \bigg\{ \frac{\delta_1}{2}, \Big(\frac{\varepsilon}{2}\Big)^{1-1/p} \frac{k^r}{(k)_r} \bigg\}.
\]
Let $H_1$ be obtained from $H$ by removing all edges that contain a pair whose codegree does not 
exceed $d:= \big(k + 1 + (r-2)\binom{k+1}{2}\big) \binom{n}{r-3}$.
Since the number of pairs of vertices is $\binom{n}{2}$, we obtain 
\begin{align}
e(H\setminus H_1) 
& \leq \bigg(k + 1 + (r-2)\binom{k+1}{2}\bigg) \binom{n}{r-3} \times \binom{n}{2} \nonumber \\
& < \Big(\frac{\delta \cdot (k)_r}{k^r}\Big)^{p/(p-1)} \binom{n}{r}. \label{eq:upper-bound-H-H1}
\end{align}

Now, we show that $H_1$ is $\mathcal{K}_{k+1}^{(r)}$-free, which has been previously 
established by Pikhurko \cite{Pikhurko2013}. Here, we provide a detailed exposition of the proof.
Suppose, to the contrary, that $H_1$ contains a copy of some $F\in \mathcal{K}_{k+1}^{(r)}$ 
with core $C$. Then, every pair from $C$ is covered by an edge of $H_1$. Consequently, 
every pair $\{u,v\}\subseteq C$ has codegree at least $d$ in $H$. This implies that whenever we have a 
partial embedding of $H_{k+1}^{(r)}$ into $H$ with the core $C$, there exists an edge $e\in E(H)$
with $u,v\in e$ such that $e\setminus\{u, v\}$ is disjoint from the rest of the embedding. 
So, $H$ contains $H_{k+1}^{(r)}$ as a subgraph with the core $C$, a contradiction.

Next, we give an estimation of $\lambda^{(p)} (H_1)$. By Proposition \ref{prop:weyl-inequality}, 
\[
\lambda^{(p)} (H) \leq \lambda^{(p)} (H_1) + \lambda^{(p)} (H\setminus H_1).
\]
On the other hand, Lemma \ref{lem:size-upper-bound} implies that
\[
\lambda^{(p)} (H\setminus H_1) < \big(r! \cdot e(H\setminus H_1)\big)^{1-1/p} < 
\frac{\delta \cdot (k)_r}{k^r} \cdot n^{r(1-1/p)},
\]
where the last inequality due to \eqref{eq:upper-bound-H-H1}. Hence, 
\begin{align*}
\lambda^{(p)} (H_1) 
& \geq \lambda^{(p)}(H) - \lambda^{(p)} (H\setminus H_1) \\
& > (1 - 2\delta) \frac{(k)_r}{k^r} n^{r(1-1/p)} \\
& \geq (1 - \delta_1) \frac{(k)_r}{k^r} n^{r(1-1/p)}.
\end{align*}
It follows from Lemma \ref{lem:spectral-stability-general-clique} that 
$H_1$ is $\frac{\varepsilon}{2} \binom{n}{r}$-close to $T_r(n,k)$.
Recall that $H_1$ is obtained from $H$ by removing at most 
\[ 
\Big(\frac{\delta \cdot (k)_r}{k^r} \Big)^{p/(p-1)} \binom{n}{r} \leq \frac{\varepsilon}{2} \binom{n}{r}
\]
edges by \eqref{eq:upper-bound-H-H1}. Hence, $H$ is 
$\varepsilon \binom{n}{r}$-close to $T_r(n,k)$. This completes the proof. \hfill\ensuremath{\Box}

\subsection{Spectral stability of $F_{k+1}^{(r)}$}

This subsection is dedicated to proving Theorem \ref{thm:spectral-stability-Fan}, which 
follows a similar approach to Theorem \ref{thm:spectral-stability-expanded-clique} but 
incorporates several distinct technical details.

Mubayi and Pikhurko \cite{Mubayi-Pikhurko2007} proved that $\mathcal{F}_{k+1}^{(r)}$
is stable with respect to $T_r(n,k)$.

\begin{lemma}[{\cite[Theorem 4]{Mubayi-Pikhurko2007}}]\label{lem:stability-for-edge-2}
For any $k\geq r\geq 2$ and $\varepsilon > 0$, there exist $\delta > 0$ and $n_0$
such that the following holds for all $n > n_0$: If $H$ is an $n$-vertex 
$\mathcal{F}_{k+1}^{(r)}$-free $r$-graph with at least $(1-\delta) \frac{(k)_r}{k^r} \binom{n}{r}$ 
edges, then $H$ is $\varepsilon\binom{n}{r}$-close to $T_r(n,k)$.
\end{lemma}

By Lemma \ref{lem:general-criterion} and Proposition \ref{prop:Lagrangian-density-1} we obtain the 
spectral stability of $\mathcal{F}_{k+1}^{(r)}$ immediately.

\begin{lemma}[Spectral Stability of $\mathcal{F}_{k+1}^{(r)}$]\label{lem:spectral-stability-Fan}
Let $p>1$ and $k\geq r\geq 2$. For any $\varepsilon > 0$, there are $\delta = \delta(k,r,\varepsilon) > 0$
and $n_0 = n_0(k,r,\varepsilon)$ such that the following holds for all $n>n_0$:
If $H$ is an $n$-vertex $\mathcal{F}_{k+1}^{(r)}$-free $r$-graph with
\[
\lambda^{(p)} (H) > (1 - \delta) \frac{(k)_r}{k^r} n^{r(1-1/p)},
\]
then $H$ is $\varepsilon \binom{n}{r}$-close to $T_r(n, k)$.
\end{lemma}

We are now prepared to present the proof of Theorem \ref{thm:spectral-stability-Fan}, which follows a 
similar approach to the proof of Theorem \ref{thm:spectral-stability-expanded-clique}.
\par\vspace{3mm}

\noindent\emph{Proof of Theorem \ref{thm:spectral-stability-Fan}}.
Let $\varepsilon > 0$ be given. By Lemma \ref{lem:spectral-stability-Fan}, there exist 
$\delta_2>0$ and $n_2$ such that of any $\mathcal{F}_{k+1}^{(r)}$-free $r$-graph $G$ 
having $n\geq n_2$ vertices and $\lambda^{(p)} (G) > (1 - \delta_2) \frac{(k)_r}{k^r} n^{r(1-1/p)}$ 
is $\frac{\varepsilon}{2} \binom{n}{r}$-close to $T_r(n,k)$. 

Choose 
\[
\delta = \min \bigg\{ \frac{\delta_2}{2}, \Big(\frac{\varepsilon}{2}\Big)^{1-1/p} \frac{k^r}{(k)_r} \bigg\}.
\]
Let $H_2$ be obtained from $H$ by removing all edges that contain a pair whose codegree at most $d$. Then
\begin{equation}\label{eq:upper-bound-H-H2}
e(H\setminus H_2) 
\leq d \times \binom{n}{2} < \bigg(\frac{\delta \cdot (k)_r}{k^r} \bigg)^{p/(p-1)} \binom{n}{r}.
\end{equation}

We now demonstrate that $H_2$ is $\mathcal{F}_{k+1}^{(r)}$-free, as previously established 
by Mubayi and Pikhurko \cite{Mubayi-Pikhurko2007}. Here, we provide a detailed proof.
Suppose, to the contrary, that $H_2$ contains a copy of some $F\in \mathcal{F}_{k+1}^{(r)}$ 
with core $C$ and edge $e\subseteq C$. Then, every pair from $C$ is covered by an edge 
of $H_2$, implying every pair in $C$ has codegree at least $d$ in $H$. Consequently, 
we can greedily choose edges of $H$ that contain all pairs in $\binom{C}{2}\setminus \binom{e}{2}$,
such that these edges intersect $C$ in exactly two vertices and are pairwise disjoint 
outside $C$. The resulting set of the $\binom{k+1}{2} - \binom{r}{2}$ edges, together 
with $e$, forms a copy of $F_{k+1}^{(r)}$ in $H$, which contradicts the assumption that 
$H$ is $F_{k+1}^{(r)}$-free.

In light of Proposition \ref{prop:weyl-inequality}, Lemma \ref{lem:size-upper-bound} 
and \eqref{eq:upper-bound-H-H2}, we find 
\begin{align*}
\lambda^{(p)} (H) 
& \leq \lambda^{(p)} (H_2) + \lambda^{(p)} (H\setminus H_2) \\
& \leq \lambda^{(p)} (H_2) + \big(r! \cdot e(H\setminus H_2)\big)^{1-1/p} \\
& < \lambda^{(p)} (H_2) + \frac{\delta \cdot (k)_r}{k^r} n^{r(1-1/p)}
\end{align*}
As a consequence, 
\begin{align*}
\lambda^{(p)} (H_2) 
& > \lambda^{(p)}(H) - \frac{\delta \cdot (k)_r}{k^r} n^{r(1-1/p)} \\
& > (1 - 2\delta) \frac{(k)_r}{k^r} n^{r(1-1/p)} \\
& \geq (1 - \delta_2) \frac{(k)_r}{k^r} n^{r(1-1/p)}.
\end{align*}
It follows from Lemma \ref{lem:spectral-stability-Fan} that 
$H_2$ is $\frac{\varepsilon}{2} \binom{n}{r}$-close to $T_r(n,k)$.
Recall that $H_2$ is obtained from $H$ by removing at most 
\[ 
\bigg(\frac{\delta \cdot (k)_r}{k^r} \bigg)^{p/(p-1)} \binom{n}{r} 
\leq \frac{\varepsilon}{2} \binom{n}{r}
\]
edges by \eqref{eq:upper-bound-H-H2}. Hence, $H$ is 
$\varepsilon \binom{n}{r}$-close to $T_r(n,k)$. This completes the proof. \hfill\ensuremath{\Box}

\section{Proof of Theorem \ref{thm:main-1}}
\label{sec4}

In this section we give a proof of Theorem \ref{thm:main-1}. Before continuing, 
we need some additional notation and preliminary results. 

Given an $r$-graph $G$ of order $n$, let $\bm{\sigma} = (V_1,V_2,\ldots,V_k)$ be 
a partition of $V(G)$, we define 
\[
f_G(\bm{\sigma}) := \sum_{e\in E(G)} |\{i\in [k]: e\cap V_i \neq \emptyset\}|.
\]
Let $T$ be the complete $k$-partite $r$-graph on $n$ vertices, partitioned according to $\bm{\sigma}$.
We call the edges in $T \setminus G$ \emph{missing} and the edges in $G \setminus T$ \emph{bad}.

For vertices $u, v$ coming from two different parts $V_i$, the pair $\{u,v\}$ is called \emph{sparse} 
if the codegree of $u$ and $v$ is at most 
\begin{equation}\label{eq:constant-d}
d:= \bigg[ k + 1 + (r-2)\binom{k+1}{2}\bigg] \binom{n}{r-3};
\end{equation}
otherwise, the pair $\{u,v\}$ is called \emph{dense}.

\subsection{Preparation for the proof of Theorem \ref{thm:main-1}}

The following result indicates that any $k$-partite $r$-graph, whose size is nearly maximal, 
must have an almost balanced partition of its vertices.

\begin{proposition}[{\cite[Claim 1]{Mubayi2006}}]\label{prop:nearly-equal-Mubayi}
Let $k\geq r\geq 2$. For any $\varepsilon > 0$ there are $\delta > 0$ and $n_0$ such 
that, for any $k$-partite $r$-graph of order $n\geq n_0$ and size at least 
$t_r(n,k) -\delta\binom{n}{r}$, the number of vertices in each part located in 
$((1/k - \varepsilon) n, (1/k + \varepsilon) n)$.
\end{proposition}

In order to finish the proof of Theorem \ref{thm:main-1}, we need the following lemma.

\begin{lemma}\label{lem:size-Vi}
Let $k\geq r\geq 3$, and $\varepsilon > 0$. Suppose that $G$ is an $r$-graph on $n$ vertices, 
and $G$ is $\varepsilon \binom{n}{r}$-close to $T_r(n,k)$. Let $\bm{\sigma} = (V_1,V_2,\ldots,V_k)$ 
be a partition of $V(G)$ such that $f_G(\bm{\sigma})$ attains the maximum. Then for each $i\in [k]$,
\[ 
\Big( \frac{1}{k} - \varepsilon^{1/r} \Big) n \leq |V_i| \leq \Big(\frac{1}{k} + \varepsilon^{1/r} \Big) n.
\]
\end{lemma}

\begin{proof}
Let $T$ be the complete $k$-partite $r$-graph on $n$ vertices with vertex partition $\bm{\sigma}$. 
By the definition of $f_G(\bm{\sigma})$, we have 
\[ 
r\bigg[ e(G) - \varepsilon\binom{n}{r} \bigg] \leq f_G(\bm{\sigma})\leq r\cdot e(G) - e(G\setminus T),
\]
from which we immediately obtain 
\begin{equation}\label{eq:size-G-T}
e(G\setminus T) \leq r\varepsilon\binom{n}{r}.
\end{equation}
It follows from \eqref{eq:size-G-T} and $e(G) \geq t_r(n, k) - \varepsilon\binom{n}{r}$ that
\begin{align}
e(T) & \geq e(G) - e(G\setminus T) \nonumber \\ 
& \geq t_r(n,k) - (r+1) \varepsilon\binom{n}{r} \nonumber \\ 
& \geq \binom{k}{r}\left(\frac{n}{k}\right)^r - (r+2)\varepsilon\binom{n}{r}.
\label{eq:e(T)-lower-bound-2}
\end{align}

Next, we will establish an upper bound for $e(T)$. Combining this bound with \eqref{eq:e(T)-lower-bound-2}, 
we shall conclude the desired assertion. To this end, we assume, without loss of generality, that $V_1$ satisfies
\[ 
\Big||V_1|-\frac{n}{k}\Big| = \max_{i\in[k]} \Big\{\Big||V_i| - \frac{n}{k} \Big|\Big\}.
\]
Set $a := \big|\frac{|V_1|}{n} - \frac{1}{k}\big|$ for short. Then $|V_1| = (1/k \pm a) n$. 
By Proposition \ref{prop:nearly-equal-Mubayi} and \eqref{eq:e(T)-lower-bound-2}, $a\ll 1$.

For $0\leq x\leq 1$, let 
\[ 
f_{k,r}(x) := x\binom{k-1}{r-1} \Big(\frac{1-x}{k-1}\Big)^{r-1} + \binom{k-1}{r} \Big(\frac{1-x}{k-1}\Big)^r.
\]
For any $c\in [n]$, one can check that
\[
f_{k,r} \Big(\frac{c}{n}\Big)\cdot n^r 
= c\binom{k-1}{r-1} \Big(\frac{n-c}{k-1}\Big)^{r-1} 
+ \binom{k-1}{r} \Big(\frac{n-c}{k-1}\Big)^r,
\]
from which we obtain that
\begin{equation}\label{eq:se4-temporary}
f_{k,r} \Big(\frac{c}{n}\Big) \cdot n^r \geq c\cdot t_{r-1}(n - c, k-1) + t_r(n - c, k-1).
\end{equation}
Indeed, for each $c$, we can consider the right-hand side of \eqref{eq:se4-temporary} as 
the number of edges in $T_r(n-c, k-1)$ together with $c$ additional vertices each of whose
links is $T_{r-1} (n-c, k-1)$. Let $c = |V_1|$ in \eqref{eq:se4-temporary}, and note that 
$T$ is a $k$-partite $r$-graph. We have 
\begin{equation}\label{eq:e(T)-upper-bound}
e(T) \leq \max\Big\{f_{k,r} \Big(\frac{1}{k} - a\Big) \cdot n^r, ~f_{k,r} \Big(\frac{1}{k} + a\Big) \cdot n^r\Big\}.
\end{equation}

In the following, we estimate $f_{k,r} (1/k - a)$ and $f_{k,r} (1/k + a)$ respectively. Note that 
\begin{align*}
f_{k,r}\left(\frac{1}{k}-a\right) 
& = \left(\frac{1}{k}-a\right)\binom{k-1}{r-1}\left(\frac{1}{k}+\frac{a}{k-1}\right) ^{r-1} 
+ \binom{k-1}{r}\left(\frac{1}{k}+\frac{a}{k-1}\right)^r \\[1mm]
& = \frac{1}{k^{r-1}} \left(\frac{1}{k} - a\right) \binom{k-1}{r-1}\left(1 + \frac{ka}{k-1}\right)^{r-1} 
+ \frac{1}{k^r} \binom{k-1}{r}\left(1 + \frac{ka}{k-1}\right)^r \\[1mm]
& = \frac{1}{k^r} \Big(1 + \frac{ka}{k-1}\Big)^{r-1} \bigg[ (1-ka) \binom{k-1}{r-1} + 
\Big(1 + \frac{ka}{k-1}\Big) \binom{k-1}{r}\bigg] \\[1mm]
& = \frac{1}{k^r} \binom{k}{r} \Big(1 + \frac{ka}{k-1}\Big)^{r-1} \Big(1 - \frac{(r-1)k}{k-1} a\Big).
\end{align*}
Since $a \ll 1$, one can check that
\[
\Big(1 + \frac{ka}{k-1}\Big)^{r-1} \leq 1 + \frac{(r-1)k}{k-1} a + \frac{(r-1)(r-2) k^2}{(k-1)^2} a^2.
\]
It follows that 
\begin{align*}
f_{k,r}\left(\frac{1}{k}-a\right) 
& \leq \frac{1}{k^r} \binom{k}{r} \bigg(1 + \frac{(r-1)k}{k-1} a + \frac{(r-1)(r-2) k^2}{(k-1)^2} a^2\bigg) 
\Big(1 - \frac{(r-1)k}{k-1} a\Big) \\[1mm]
& \leq \frac{1}{k^r} \binom{k}{r} \bigg(1 - \frac{(r-1)k^2}{(k-1)^2} a^2\bigg).
\end{align*}
Similarly, we have 
\[ 
f_{k,r}\left(\frac{1}{k}+a\right) \leq \frac{1}{k^r} \binom{k}{r} 
\bigg(1 - \frac{(r-1)k^2}{(k-1)^2} a^2\bigg).
\] 

Finally, combining these with \eqref{eq:e(T)-lower-bound-2} and \eqref{eq:e(T)-upper-bound}, 
we conclude that
\[ 
\binom{k}{r}\left(\frac{n}{k}\right)^r - (k+2)\varepsilon\binom{n}{r} \leq 
\frac{1}{k^r} \binom{k}{r} \bigg(1 - \frac{(r-1)k^2}{(k-1)^2} a^2\bigg) n^r.
\]
Solving the inequality we deduce that 
$a^2 \leq \frac{k^{r+1}}{r!} \varepsilon \leq \varepsilon^{2/r}$, which completes the proof.
\end{proof}

\subsection{Proof of Theorem \ref{thm:main-1}}

For the remainder of this section, we always assume that $k\geq r\geq 2$, $p\geq r$, and $H$ is 
an $n$-vertex $H_{k+1}^{(r)}$-free $r$-graph attaining the maximum $p$-spectral 
radius. We will show that $H\cong T_r(n,k)$ through the following lemmas, thereby 
concluding the proof of Theorem \ref{thm:main-1}. When $r = 2$, Theorem \ref{thm:main-1} 
is exact the spectral Tur\'an theorem \cite{Kang-Nikiforov2014}. We therefore assume 
that $k\geq r > 2$. 

The proof of Theorem \ref{thm:main-1} is outlined as follows.
\begin{itemize}\setlength{\itemsep}{-2pt}
\item We start by showing that $H$ is $\varepsilon\binom{n}{r}$-close to $T_r(n,k)$ via the 
spectral stability theorem for $H_{k+1}^{(r)}$ (Theorem \ref{thm:spectral-stability-expanded-clique}).
Furthermore, $H$ has a large multipartite subgraph on parts $V_1,\ldots,V_k$ such that
$|V_i| = (1 + o(1)) n/k$ for each $i\in [k]$ (Corollary \ref{coro:equal-size-Vi}).

\item We show that the number of vertices that contained in at least $\varepsilon^{1/r^2} n$ 
sparse pairs (denoted the set of these vertices by $L$) is bounded from above by $o(n)$ 
(Lemma \ref{lem:size-L}). By removing the vertices in $L$ we show that $V_i\setminus L$ 
is a strong independent set of $H$ for each $i\in [k]$ (Lemma \ref{lem:independent-set}).

\item Using Corollary \ref{coro:equal-size-Vi}, Lemma \ref{lem:size-L}, and Lemma \ref{lem:independent-set}, 
we show that for any vertex $u$, the size of set $\{e\in E_u: e\cap L\neq\emptyset\}$ is 
bounded from above by $o(n^{r-1})$ (Lemma \ref{lem:upper-bound-L(u)}), and the degree of 
each vertex in $L$ is not excessively large (Lemma \ref{lem:degree-vertex-in-L}).

\item Finally, we will finish the proof by showing $L = \emptyset$. This can be achieved by 
introducing a modified Zykov's symmetrization process, in conjunction with previous lemmas.
\end{itemize}

Since $T_r(n,k)$ is $H_{k+1}^{(r)}$-free, by Lemma \ref{lem:lambdaT}, 
$\lambda^{(p)} (H) > (1-O(n^{-1})) \cdot \frac{(k)_r}{k^r} n^{r(1-1/p)}$. Together with 
Theorem \ref{thm:spectral-stability-expanded-clique}, $H$ is $\varepsilon\binom{n}{r}$-close 
to $T_r(n,k)$. In our proof, we always tacitly assume that $\varepsilon$ is sufficiently 
small and $n$ is large enough so that the associated inequalities hold.

Let $\bm{\sigma} = (V_1,V_2,\ldots,V_k)$ 
be a partition of $V(H)$ such that $f_H(\bm{\sigma})$ attains the maximum.
An immediate corollary of Lemma \ref{lem:size-Vi} is as follows.

\begin{corollary}\label{coro:equal-size-Vi}
For each $i\in [k]$,
\[
\Big( \frac{1}{k} - \varepsilon^{1/r} \Big) n \leq |V_i| \leq \Big(\frac{1}{k} + \varepsilon^{1/r} \Big) n.
\]
\end{corollary}

Let $L$ be the set of vertices which contained in at least $\varepsilon^{1/r^2} n$ sparse pairs of $H$. 
Denote by $W$ the set of all sparse pairs in $H$. The size of $W$ and $L$ can be estimated as follows.

\begin{lemma}\label{lem:size-L}
$|W| < \varepsilon^{1/r} n^2/2$ and $|L| < \varepsilon^{\frac{r-1}{r^2}} n$.
\end{lemma}

\begin{proof}
Let $T$ be the complete $k$-partite $r$-graph on $n$ vertices with vertex partition $\bm{\sigma}$. 
We first give an upper bound on $|W|$. To this end, note that each missing edge contains at most 
$\binom{r}{2}$ sparse pairs, and each sparse pair belongs to at least 
$\binom{k-2}{r-2} (1/k - \varepsilon^{1/r})^{r-2} n^{r-2}-d$ missing edges from Corollary \ref{coro:equal-size-Vi}, 
we have 
\[ 
|W|\left[\binom{k-2}{r-2} \Big(\frac{1}{k} - \varepsilon^{1/r}\Big)^{r-2} n^{r-2} - d\right]
\leq \binom{r}{2} e(T\setminus H).
\]
Recall that $d = \big(\binom{k+1}{2}(r-2)+k+1\big) \binom{n}{r-3}$ and by Bernoulli inequality, we see
\begin{equation}\label{eq:sec4-temporary}
|W| \left[\frac{1}{k^{r-2}} \binom{k-2}{r-2} (1 - kr\varepsilon^{1/r}) n^{r-2}\right] < \binom{r}{2} e(T\setminus H).
\end{equation}
Since $H$ is $\varepsilon\binom{n}{r}$-close to $T_r(n,k)$, $e(H) \geq t_r(n,k) - \varepsilon\binom{n}{r}$. 
It follows that $e(T) \leq e(H) + \varepsilon\binom{n}{r}$. As a consequence, 
\begin{align}\label{eq:size-T-H}
e(T\setminus H) & = e(T) - e(T\cap H) \nonumber \\
& = e(T) - (e(H) - e(H\setminus T)) \nonumber \\
& = (e(T) - e(H)) + e(H\setminus T) \nonumber \\
& \leq (r+1) \varepsilon\binom{n}{r},
\end{align}
where the last inequality follows from the fact that $e(H\setminus T) \leq r\varepsilon\binom{n}{r}$,
derived using the same arguments as \eqref{eq:size-G-T}. In light of \eqref{eq:sec4-temporary} 
and \eqref{eq:size-T-H}, we see
\begin{equation}\label{eq:size-W}
|W| < \frac{(r + 1) k^{r-2} \binom{r}{2} \binom{n}{r} \varepsilon}{\binom{k-2}{r-2} (1 - kr\varepsilon^{1/r}) n^{r-2}} 
< \frac{1}{2} \varepsilon^{1/r} n^2.
\end{equation}

Finally, recall that $L$ is the set of vertices which contained in at least $\varepsilon^{1/r^2} n$ 
sparse pairs of $H$. Hence, $|L| \varepsilon^{1/r^2} n \leq 2|W| < \varepsilon^{1/r} n^2$, which 
yields that $|L| < \varepsilon^{\frac{r-1}{r^2}} n$.
\end{proof}
 
\begin{lemma}\label{lem:independent-set}
For each $i\in[k]$, $V_i\setminus L$ is a strong independent set of $H$.
\end{lemma}

\begin{proof}
Suppose to the contrary that, without loss of generality, there exists a pair $\{u,v\}\subseteq V_1\setminus L$ 
contained in a bad edge of $H$. Let $DV_i(u)$ denote the set of vertices in $V_i$ adjacent to $u$ by dense pairs.
By the choice of $L$, $|DV_i(u)|\geq |V_i| - \varepsilon^{1/r^2} n$ and $|DV_i(v)|\geq |V_i| - \varepsilon^{1/r^2} n$. 
It follows from Lemma \ref{lem:intersection-sets} that 
\[
|\big( DV_2(u)\cap DV_2(v)\big) \setminus L| 
\geq 2\big( |V_2| - \varepsilon^{1/r^2} n\big) - |V_2| - \varepsilon^{\frac{r-1}{r^2}} n > 1.
\]
Thus we can find a vertex $w_2\in V_2\setminus L$ such that $\{w_2,u\}$ and $\{w_2,v\}$ are both dense 
pairs. For an integer $s$ with $2 \leq s \leq k-1$, suppose that there are vertices 
$\{w_2,\ldots,w_s\}\subseteq V(H)\setminus L$ such that 
$\{w_i,u\}$, $\{w_i,v\}$ ($2\leq i\leq s$) and $\{w_i,w_j\}$ ($2\leq i<j\leq s$) are both dense pairs. 
We next show the existence of a vertex in $V_{s+1}\setminus L$ that is adjacent to both the above $s+1$ vertices 
by dense pairs. It follows from Lemma \ref{lem:intersection-sets} that
\begin{align*}
& ~ |\big( DV_{s+1}(u)\cap DV_{s+1}(v)\cap DV_{s+1}(w_2)\cap\cdots \cap DV_{s+1}(w_s)\big)\setminus L| \\
\geq & ~ (s+1)\big( |V_{s+1}| - \varepsilon^{1/r^2} n\big) - s|V_{s+1}| - \varepsilon^{\frac{r-1}{r^2}} n \\
> & ~ 1.
\end{align*}
Therefore, there is a vertex $w_{s+1}\in V_{s+1}\setminus L$ such that $\{w_{s+1},u\}$, $\{w_{s+1},v\}$ 
and $\{w_i,w_{s+1}\}$ ($2\leq i\leq s$) are both dense pairs.

By the choice of $\{u,v,w_2,\ldots,w_k\}$, these vertices form a core of $H_{k+1}^{(r)}$ in $H$, which 
implies that $H_{k+1}^{(r)}\subseteq H$, a contradiction completing the proof
of Lemma \ref{lem:independent-set}.
\end{proof}

To finish the proof of Theorem \ref{thm:main-1}, it is enough to show $L = \emptyset$.
For each $u\in V(H)$, let $L(u)$ denote the set of edges containing $u$ and intersecting $L\setminus\{u\}$ non-empty. 

\begin{lemma}\label{lem:upper-bound-L(u)}
For each $u\in V(H)$, $|L(u)| < \varepsilon^{3/(2r^2)} n^{r-1}$.    
\end{lemma}

\begin{proof}
Let $u$ be any vertex of $H$. By the definition of $L(u)$, we see 
\begin{align*}
|L(u)| & \leq \sum_{i=1}^{r-1} \binom{|L|}{i} \binom{n-1-i}{r-1-i} \\
& < \sum_{i=1}^{r-1} \frac{|L|^i \cdot n^{r-1-i}}{i! (r-1-i)!} \\
& = \frac{n^{r-1}}{(r-1)!} \sum_{i=1}^{r-1} \Big(\frac{|L|}{n}\Big)^i \binom{r-1}{i} \\
& = \frac{n^{r-1}}{(r-1)!} \Big[\Big(1 + \frac{|L|}{n}\Big)^{r-1} - 1\Big].
\end{align*}
In view of $|L| < \varepsilon^{\frac{r-1}{r^2}} n$ by Lemma \ref{lem:size-L} we have 
\[
\Big(1 + \frac{|L|}{n}\Big)^{r-1} - 1 < \Big(1 + \varepsilon^{\frac{r-1}{r^2}}\Big)^{r-1} - 1
< 2(r-1) \varepsilon^{\frac{r-1}{r^2}}.
\]
Combining these inequalities above, we obtain
$|L(u)| < \varepsilon^{3/(2r^2)} n^{r-1}$, completing the proof of Lemma \ref{lem:upper-bound-L(u)}.
\end{proof}

In what follows, we set $\ell := \varepsilon^{3/(2r^2)} n^{r-1}$ for short.

\begin{lemma}\label{lem:degree-vertex-in-L}
For each $u\in L$, $d_H(u) < \binom{k-1}{r-1}\left(\frac{n}{k}\right)^{r-1} - 2\ell$.
\end{lemma}

\begin{proof}
For any $u\in L$, we assume, without loss of generality, that $u\in V_1$, i.e., 
$u\in V_1\cap L$. We proceed by considering the following two cases.

{\bfseries Case 1.} There is a vertex $v\in V_1\setminus L$ such that $\{u,v\}$ contained 
in a bad edge in $H$. 
    
We first show that there is some $j\in \{2,\ldots,k\}$ such that $|DV_j(u)|\leq k\varepsilon^{1/r^2} n$.
Otherwise, we have $|DV_i(u)| > k\varepsilon^{1/r^2} n$ for each $2\leq i\leq k$. As the same 
discussion in Lemma \ref{lem:independent-set}, we can find $k-1$ vertices $w_2,\ldots,w_k$ such that
each pair in $\{w_2,\ldots,w_k,u,v\}$, except for $\{u,v\}$, forms a dense pair. Consequently,
the $k+1$ vertices forms the core of a copy of $H_{k+1}^{(r)}$ in $H$, a contradiction.
Without loss of generality, we assume that $|DV_2(u)|\leq k\varepsilon^{1/r^2} n$.

To estimate $d_H(u)$, note that $E_u$ is a subset of the union of the following four sets:

(i) $L(u)$ (the set of edges containing $u$ and intersecting $L\setminus\{u\}$ non-empty); 

(ii) The set of edges containing $u$ and intersecting $V_2$ non-empty; 

(iii) The set of edges containing $u$ and intersecting $V_1\setminus\{u\}$ non-empty; 

(iv) The set of edges containing $u$ and intersecting $(V_1\cup V_2\cup L)\setminus\{u\}$ empty. \\
In what follows, we estimate the size of these four sets, respectively. 
Obviously, $|L(u)| < \ell$ by Lemma \ref{lem:upper-bound-L(u)}.
For (ii), by Lemma \ref{lem:independent-set} and Corollary \ref{coro:equal-size-Vi}, the number 
of edges containing $u$ and intersecting $V_2$ non-empty is at most
\begin{align*}
t & := |DV_2(u)| \binom{k-2}{r-2}\left(\frac{1}{k} + \varepsilon^{1/r}\right)^{r-2} n^{r-2}
 + |DV_2(u)| \sum_{i=1}^{r-2} \binom{|L|}{i} \binom{n - 2 - i}{r - 2 - i} \\
& ~~~~ + \Big[ \Big(\frac{1}{k} + \varepsilon^{1/r}\Big) n - |DV_2(u)| \Big] d,
\end{align*}
where the first term gives an upper bound on the number of edges in $E_u\cap E(T)$ that contain  
a vertex in $DV_2(u)$, while intersecting $L\setminus (V_1\cup V_2)$ empty; the second term provides  
an upper bound on the number of edges in $E_u$ that contain a vertex in $DV_2(u)$ and intersect
$L\setminus (V_1\cup V_2)$ non-empty; the third term offers an upper bound on the number of edges 
in $E_u$ that contain a vertex $w$ in $V_2\setminus DV_2(u)$ such that $\{u,w\}$ is a sparse pair.
Since $|DV_2(u)| \leq k\varepsilon^{1/r^2} n$, we have 
\begin{align*}
t & \leq k\varepsilon^{1/r^2} n \bigg[\binom{k-2}{r-2}\left(\frac{1}{k} + \varepsilon^{1/r}\right)^{r-2}n^{r-2}
 + \sum_{i=1}^{r-2} \binom{|L|}{i} \binom{n - 2 - i}{r - 2 - i}\bigg] \\
& ~~~~ + \Big[ \Big(\frac{1}{k} + \varepsilon^{1/r}\Big) n - k\varepsilon^{1/r^2} n \Big] d.
\end{align*}
In order to estimate $t$, note that 
\begin{align*}
\sum_{i=1}^{r-2} \binom{|L|}{i} \binom{n - 2 - i}{r - 2 - i}
& < \sum_{i=1}^{r-2} \frac{|L|^i \cdot n^{r-2-i}}{i!\cdot (r-2-i)!} \\
& = \frac{n^{r-2}}{(r-2)!} \sum_{i=1}^{r-2} \Big( \frac{|L|}{n} \Big)^i \binom{r-2}{i} \\
& = \frac{n^{r-2}}{(r-2)!} \bigg[\Big(1 + \frac{|L|}{n}\Big)^{r-2} - 1\bigg] \\
& < \frac{2\varepsilon^{(r-1)/r^2}}{(r-3)!} \cdot n^{r-2},
\end{align*}
where the last inequality due to Lemma \ref{lem:size-L}. Recall that 
$d = \big(k + 1 + (r-2) \binom{k+1}{2}\big) \binom{n}{r-3}$
by \eqref{eq:constant-d}. We have 
\[
\Big[ \Big(\frac{1}{k} + \varepsilon^{1/r}\Big) n - k\varepsilon^{1/r^2} n \Big] d = O(n^{r-2}).
\]
Combining these inequalities above together, we obtain 
\begin{align*}
t & < k\varepsilon^{1/r^2} n \bigg[\binom{k-2}{r-2}\left(\frac{1}{k} + \varepsilon^{1/r}\right)^{r-2}n^{r-2}
 + \frac{2\varepsilon^{(r-1)/r^2}}{(r-3)!} \cdot n^{r-2}\bigg] + O(n^{r-2}) \\
& < \frac{2\varepsilon^{1/r^2}}{k^{r-3}} \binom{k-2}{r-2} n^{r-1}.
\end{align*}
For (iii), we show that the number of edges containing $u$ and intersecting $V_1\setminus\{u\}$ 
non-empty is at most $2t$. Assume by contradiction that there exist at least $2t+1$ edges 
containing $u$ and intersecting $V_1\setminus\{u\}$ non-empty in $H$. Let $V'_1=V_1\setminus \{u\}$, $V'_2=V_2\cup \{u\}$ 
and $\bm{\sigma}' = V'_1\cup V'_2 \cup V_3\cup \cdots\cup V_k$ be a partition of $V(H)$. Then each edge 
containing $u$ and intersecting $V_2$ may decrease its contribution to $f_H(\bm{\sigma}')$ by $1$, 
thus the total reduced value of $f_H(\bm{\sigma}')$ corresponding to the new partition is at most $t$.
On the other hand, each edge containing $u$, intersecting $V_1\setminus\{u\}$ and disjoint from $V_2$ may 
increase its contribution to $f_H(\bm{\sigma}')$ by $1$, thus the total increased value of $f_H(\bm{\sigma}')$ 
corresponding to the new partition is at least $2t+1-t=t+1$. Hence, $f_H(\bm{\sigma}') > f_H(\bm{\sigma})$, 
this is a contradiction to the choice of the partition $V_1\cup V_2 \cup \cdots\cup V_k$.

So, in light of Corollary \ref{coro:equal-size-Vi} and Lemma \ref{lem:independent-set}, we deduce that
\begin{align*}
d_H(u) 
& \leq \binom{k-2}{r-1} 
\Big(\frac{1}{k} + \varepsilon^{1/r}\Big)^{r-1} n^{r-1} 
+ 3t + \ell \\
& < \binom{k-2}{r-1}\left(\frac{n}{k}\right)^{r-1} (1 + 2(r-1) k\varepsilon^{1/r}) 
+ \frac{6\varepsilon^{1/r^2}}{k^{r-3}} \binom{k-2}{r-2} n^{r-1} + \ell \\
& < \binom{k-1}{r-1} \left(\frac{n}{k}\right)^{r-1} - 2\ell.
\end{align*}

{\bfseries Case 2.} For any vertex $v\in V_1\setminus L$, $\{u,v\}$ is not covered by a bad edge. 
    
In this case, by inclusion-exclusion principle, the number of edges in $H$ containing $u$ is at most
\begin{equation}\label{eq:upper-bound-d(u)-L}
\binom{k-1}{r-1} \left(\frac{1}{k} + \varepsilon^{1/r}\right)^{r-1} n^{r-1} + \ell
- \frac{\varepsilon^{1/r^2} n}{r-1} \bigg[\binom{k-2}{r-2} \left(\frac{1}{k} 
- \varepsilon^{1/r}\right)^{r-2} n^{r-2} - d\bigg],
\end{equation}
where the first term gives an upper bound on the number of edges in $E_u\cap E(T)$ that intersect 
$L\setminus \{u\}$ empty; the second term provides an upper bound on the number of edges in $E_u$ 
that intersect $L\setminus \{u\}$ non-empty. It follows from \eqref{eq:upper-bound-d(u)-L} that 
\begin{align*}
d(u) & < \binom{k-1}{r-1} \left(1+ k\varepsilon^{1/r}\right)^{r-1} \Big(\frac{n}{k}\Big)^{r-1} %
- \varepsilon^{1/r^2} \binom{k-1}{r-1} \frac{(1-k\varepsilon^{1/r})^{r-2}}{(k-1)k^{r-2}} n^{r-1} + \ell + O(n^{r-2}) \\
& < \Big[(1 + k\varepsilon^{1/r})^{r-1} - \varepsilon^{1/r^2} (1 - k\varepsilon^{1/r})^{r-2}\Big] \binom{k-1}{r-1} %
\Big(\frac{n}{k}\Big)^{r-1} + \ell + O(n^{r-2}) \\
& < \binom{k-1}{r-1} \left(\frac{n}{k}\right)^{r-1} - 2\ell.
\end{align*}
The last inequality holds because $\varepsilon^{1/r} \ll \varepsilon^{3/(2r^2)} \ll \varepsilon^{1/r^2}$ 
and $\ell = \varepsilon^{3/(2r^2)} n^{r-1}$.
\end{proof}

\begin{lemma}\label{lem:x>0-L}
If $L\neq \emptyset$, then there exists a vertex $u\in L$ such that $x_u > 0$. 
\end{lemma}

\begin{proof}
Suppose by contradiction that $x_v = 0$ for each $v\in L$. Let $H' := H\setminus L$. Then 
\[
\lambda^{(p)} (H) = r! \sum_{e\in E(H)} \prod_{u\in e} x_u 
= r! \sum_{e\in E(H')} \prod_{u\in e} x_u \leq \lambda^{(p)} (H').
\]
In view of Lemma \ref{lem:independent-set}, $H'$ is a $k$-partite $r$-graph. 
Combining with Lemma \ref{lem:k-partite-upper-bound}, we see 
\[ 
\lambda^{(p)} (H) \leq \lambda^{(p)} (H') \leq \lambda^{(p)} (T_r(n-|L|, k)) < \lambda^{(p)} (T_r(n, k)),
\]
the last inequality follows from Lemma \ref{lem:complete-k-partite-x-positive} and 
Lemma \ref{lem:subgraph-spectral-radius}. This is a contradiction with 
$\lambda^{(p)} (H) \geq \lambda^{(p)} (T_r(n,k))$.
\end{proof}

\begin{lemma}\label{lem:degree-of-z}
Let $z$ be a vertex such that $x_z = \max_{w\in V(H)}\{x_w\}$. Then 
\[
d_H(z) > \binom{k-1}{r-1}\left(\frac{n}{k}\right)^{r-1} - 2\ell.
\]
\end{lemma}

\begin{proof}
By eigenvalue\,--\,eigenvector equation \eqref{eq:eigenequation} for $\lambda^{(p)} (H)$ with 
respect to $z$, we see
\[ 
\lambda^{(p)} (H)\cdot x_z^{p-1} = (r-1)! \sum_{e\in E_z} \prod_{i\in e\setminus\{z\}} x_i 
\leq (r-1)! \cdot d_H(z) x_z^{r-1}.
\]
Combining with Lemma \ref{lem:lambdaT}, we have 
\begin{equation}\label{eq:d(z)-lower-bound}
d_H(z) \geq \frac{\lambda^{(p)} (H)}{(r - 1)!} \cdot x_z^{p-r}
> \big(1 - O(n^{-1})\big) \cdot \frac{\binom{k-1}{r-1} n^{r(1-1/p)}}{k^{r-1}} \cdot x_z^{p-r}.
\end{equation}
It follows from \eqref{eq:d(z)-lower-bound} and $x_z \geq n^{-1/p}$ that
\begin{align*}
d_H(z) & > \big(1 - O(n^{-1})\big) \cdot \frac{\binom{k-1}{r-1} n^{r-1}}{k^{r-1}} > 
\binom{k-1}{r-1}\left(\frac{n}{k}\right)^{r-1} - 2\ell.
\end{align*}
This completes the proof of Lemma \ref{lem:degree-of-z}.
\end{proof}

\begin{lemma}\label{lem:L-emptyset}
$L = \emptyset$.
\end{lemma}

\begin{proof}
Suppose to the contrary that $L\neq\emptyset$. By Lemma \ref{lem:x>0-L}, we may choose a vertex $u$ 
in $L$ such that $x_u > 0$. Let $z$ be a vertex such that $x_z = \max_{w\in V(H)}\{x_w\}$. We assume 
without loss of generality that $z\in V_1$. Moreover, by Lemma \ref{lem:degree-of-z} and 
Lemma \ref{lem:degree-vertex-in-L} we deduce that $z\notin L$.

Let $SE_z$ be the subset of $E_z$ such that each edge in $SE_z$ contains at least one sparse pair. Then
\begin{equation}\label{eq:size-SE-z}
|SE_z|\leq |W|\cdot d 
\leq \frac{\varepsilon^{1/r}}{2} n^2 \left(\binom{k+1}{2} (r-2) + k + 1\right) \binom{n}{r-3} 
< \frac{\varepsilon^{3/(2r^2)}}{2} n^{r-1}.
\end{equation}
Let $E_z^c=E_z\setminus (SE_z\cup L(z))$. 
Again, using eigenvalue\,--\,eigenvector equation for $\lambda^{(p)} (H)$ with respect to $z$, we obtain
\[
\lambda^{(p)} (H)\cdot x_z^{p-1}  
\leq (r-1)! \sum_{e\in E_z^c} \prod_{i\in e\setminus\{z\}} x_i
+ (r-1)! \sum_{e\in SE_z\cup L(z)} \prod_{i\in e\setminus\{z\}} x_i.
\]
In light of \eqref{eq:size-SE-z} and Lemma \ref{lem:upper-bound-L(u)},
\begin{align*}
\sum_{e\in E_z^c} \prod_{i\in e\setminus\{z\}} x_i
& \geq \frac{\lambda^{(p)} (H)}{(r-1)!} \cdot x_z^{p-1} 
- \sum_{e\in SE_z\cup L(z)} \prod_{i\in e\setminus\{z\}} x_i \\
& \geq \bigg( \frac{\lambda^{(p)} (H)}{(r-1)!} \cdot x_z^{p-r} -
(|SE_z| + |L(z)|) \bigg) x_z^{r-1} \\
& > \bigg ( \frac{\lambda^{(p)} (H)}{(r-1)!} \cdot x_z^{p-r} - \frac{3}{2} \ell\bigg) x_z^{r-1}.
\end{align*}
By Lemma \ref{lem:lambdaT}, we have
\[
\lambda^{(p)} (H) > \frac{(r-1)!}{k^{r-1}} \bigg(1 - O\Big(\frac{1}{n}\Big)\bigg) \binom{k-1}{r-1} n^{r(1-1/p)}.
\]
As a consequence,
\begin{equation}\label{eq:inequality-last-step}
\sum_{e\in E_z^c} \prod_{i\in e\setminus\{z\}} x_i
\geq  \bigg[\binom{k-1}{r-1} \left(\frac{n}{k}\right)^{r-1} - \frac{7}{4} \ell\bigg] x_z^{r-1}.
\end{equation}

Finally, we shall construct a new $r$-graph $G$ on $n$ vertex that is $H_{k+1}^{(r)}$-free but has larger 
$p$-spectral radius, leading to a contradiction and thus completing the proof of the lemma. To this end,
let $SE$ be the set of edges in $H$ containing at least one sparse pair, and $E_z^u=\{(e\setminus\{z\})\cup \{u\}: e\in E_z^c\}$.
Construct an $r$-graph $G$ with $V(G) = V(H)$ and 
\[ 
E(G) = \big( E(H)\setminus E_u\big) \cup E_z^u.
\]

To show that $G$ is $H_{k+1}^{(r)}$-free, assume for contradiction that $H_{k+1}^{(r)}\subseteq G$. 
This implies that the copy of $H_{k+1}^{(r)}$ in $G$ must contain vertex $u$. If $u$ is not contained 
in the core of $H_{k+1}^{(r)}$, then according to the construction of $G$, there is a copy of $H_{k+1}^{(r)}$ 
in $G$ that does not contain $u$, a contradiction. Thus, $u$ is contained in the core of $H_{k+1}^{(r)}$. 
It follows that $z$ is not contained in the core. Let $\{u,w_1,w_2,\ldots,w_k\}$ be the core of $H_{k+1}^{(r)}$. 
By the construction of $G$, $\{w_1,w_2,\ldots,w_k\}\subseteq (V_2\cup \cdots\cup V_k)\setminus L$. 
By the definition of $L$, there is a vertex $v\in V_1\setminus(L\cup\{z,u\})$ such that for $i\in [k]$ 
each $\{v, w_i\}$ is dense pair. Therefore, we also can find a copy of $H_{k+1}^{(r)}$ in $H$, a contradiction.

To show that $G$ has larger $p$-spectral radius, note that
\begin{align*}
\lambda^{(p)} (G) - \lambda^{(p)} (H)
& \geq r! \Bigg( \sum_{e\in E_z^u} \prod_{i\in e} x_i - \sum_{e\in E_u} \prod_{i\in e} x_i \bigg) \\
& = r! \Bigg( \sum_{e\in E_z^c} \prod_{i\in (e\setminus \{z\}) \cup \{u\}} x_i - \sum_{e\in E_u} \prod_{i\in e} x_i \Bigg) \\
& = r! x_u \Bigg( \sum_{e\in E_z^c} \prod_{i\in e\setminus \{z\}} x_i - \sum_{e\in E_u} \prod_{i\in e\setminus \{u\}} x_i \Bigg) \\
& > r! x_u \Bigg( \sum_{e\in E_z^c} \prod_{i\in e\setminus \{z\}} x_i - d_H(u) x_z^{r-1}\Bigg).
\end{align*}
By \eqref{eq:inequality-last-step} and Lemma \ref{lem:degree-vertex-in-L}, we deduce that
\begin{align*}
\lambda^{(p)} (G) - \lambda^{(p)} (H)
& > r! x_u \left[ \bigg(\binom{k-1}{r-1}\left(\frac{n}{k}\right)^{r-1} - \frac{7}{4} \ell \bigg) x_z^{r-1} \right. \\
& \phantom{> r! x_u [} \left. - \bigg( \binom{k-1}{r-1} \left(\frac{n}{k}\right)^{r-1} - 2\ell \bigg) x_z^{r-1} \right] \\
& = \frac{r! \ell}{4} x_u x_z^{r-1} > 0.
\end{align*}
This contradicts the fact that $H$ has the maximum $p$-spectral radius over all
$H_{k+1}^{(r)}$-free $r$-graphs.
\end{proof}

\noindent\emph{Proof of Theorem \ref{thm:main-1}}.
By applying Lemma \ref{lem:independent-set} and Lemma \ref{lem:L-emptyset}, we 
conclude that $H$ is a subgraph of $T$. Given the maximality of $\lambda^{(p)} (H)$, 
it follows that $H\cong T_r(n,k)$ from Lemma \ref{lem:k-partite-upper-bound}. \hfill $\Box$

\section{Proof of Theorem \ref{thm:main-2}}
\label{sec5}

The aim of this section is to prove Theorem \ref{thm:main-2}. While the proof follows a similar 
approach to that of Theorem \ref{thm:main-1}, it involves some distinct technical details.

In this section, we always assume that $k\geq r\geq 2$, $p\geq r$, and $H$ is an $n$-vertex 
$F_{k+1}^{(r)}$-free $r$-graph attaining the maximum $p$-spectral radius. We will show that 
$H\cong T_r(n,k)$, thereby concluding the proof of Theorem \ref{thm:main-2}. When $r = 2$, 
Theorem \ref{thm:main-2} is exact the spectral Tur\'an theorem \cite{Kang-Nikiforov2014} for triangle.
We, therefore assume that $k\geq r > 2$.

Since $T_r(n,k)$ is $F_{k+1}^{(r)}$-free, by Lemma \ref{lem:lambdaT}, 
$\lambda^{(p)} (H) > (1-O(n^{-1})) \cdot \frac{(k)_r}{k^r} n^{r(1-1/p)}$. 
Together with Theorem \ref{thm:spectral-stability-Fan}, $H$ is $\varepsilon\binom{n}{r}$-close 
to $T_r(n,k)$. Let $\bm{\sigma} = (V_1,V_2,\ldots,V_k)$ be a partition of $V(H)$ such that 
$f_H(\bm{\sigma})$ attains the maximum. Hence, for each $i\in [k]$,
\begin{equation}\label{eq:equal-size-Vi-Fan}
\Big( \frac{1}{k} - \varepsilon^{1/r} \Big) n \leq |V_i| \leq \Big(\frac{1}{k} + \varepsilon^{1/r} \Big) n.
\end{equation}

Let us define $T$ as the complete $k$-partite $r$-graph on $n$ vertices, partitioned according 
to $\bm{\sigma}$. Let $M$ denote the collection of vertices, each of which incidents with at 
least $\varepsilon^{5/4r^2}n^{r-1}$ missing edges, and let $L$ be the set defined in Section \ref{sec4}. 

\begin{lemma}\label{lem:M-size}
$|M| < \varepsilon^{\frac{r-1}{r^2}} n$.
\end{lemma}

\begin{proof}
Since each missing edge contains at most $r$ vertices in $M$, by \eqref{eq:size-T-H-F} we find
\[ 
|M|\varepsilon^{5/4r^2}n^{r-1}\leq r(r+1)\varepsilon\binom{n}{r}.
\]
This yields that 
\begin{equation}\label{eq:size-M-F}
|M|\leq \frac{r+1}{(r-1)!} \varepsilon^{1-\frac{5}{4r^2}} n < \varepsilon^{\frac{r-1}{r^2}} n,
\end{equation}
as desired.
\end{proof}

\begin{lemma}\label{lem:L-subset-M}
$L\subseteq M$.
\end{lemma}

\begin{proof}
Recall that each missing edge containing $u$ has at most $r-1$ sparse pairs. On the other hand, 
\eqref{eq:equal-size-Vi-Fan} implies that each sparse pair belongs to at least 
$\binom{k-2}{r-2} (1/k - \varepsilon^{1/r})^{r-2} n^{r-2}-d$ missing edges. Thus, for each 
vertex $u\in L$, the number of missing edges containing $u$ is at least 
\[ 
\frac{1}{r-1}\varepsilon^{1/r^2} n\left(\binom{k-2}{r-2} \Big(\frac{1}{k} - \varepsilon^{1/r}\Big)^{r-2} n^{r-2} - d\right) 
> \frac{\varepsilon^{1/r^2}}{r!\cdot (r-1)k^r} n^{r-1}
> \varepsilon^{5/4r^2} n^{r-1}.
\]
It follows from the definition of $M$ that $L\subseteq M$.
\end{proof}

\begin{lemma}\label{lem:independent-set-F}
For each $i\in[k]$, $V_i\setminus M$ is a strong independent set of $H$.
\end{lemma}

\begin{proof}
Suppose to the contrary that, without loss of generality, there exists a pair $\{u,v\}\subseteq ((V_1\setminus M)\cap e)$, 
where $e$ is a bad edge of $H$. Let $DV_i(u)$ denote the set of vertices in $V_i$ adjacent to $u$ by dense pair.
By the definition of $L$ and $L\subseteq M$ from Lemma \ref{lem:L-subset-M}, $|DV_i(u)|\geq |V_i| - \varepsilon^{1/r^2} n$ 
and $|DV_i(v)|\geq |V_i| - \varepsilon^{1/r^2} n$. It follows from Lemma \ref{lem:intersection-sets} that 
\begin{align*}
\left|\big( DV_2(u)\cap DV_2(v)\big) \setminus (M\cup e)\right| 
& \geq 2\big( |V_2| - \varepsilon^{1/r^2} n\big) - |V_2| - \varepsilon^{\frac{r-1}{r^2}} n -(r-2) \\
& > \Big(\frac{1}{k} - \varepsilon^{1/r^3}\Big) n.
\end{align*}
Thus we can find a vertex $w_2\in V_2\setminus L$ such that $\{w_2,u\}$ and $\{w_2,v\}$ are both dense pairs. 
Moreover, the number of ways to choose such  $w_2$ is at least $(1/k - \varepsilon^{1/r^3}) n$. 
For an integer $s$ with $2 \leq s \leq k-1$, suppose that there are vertices $\{w_2,\ldots,w_s\}\subseteq V(H)\setminus L$ 
such that $\{w_i,u\}$, $\{w_i,v\}$ ($2\leq i\leq s$) and $\{w_i,w_j\}$ ($2\leq i<j\leq s$) are both dense pairs. 
We have
\begin{align*}
& ~ |\big( DV_{s+1}(u)\cap DV_{s+1}(v)\cap DV_{s+1}(w_2)\cap\cdots \cap DV_{s+1}(w_s)\big)\setminus (M\cup e)| \\
\geq & ~ (s+1)\big( |V_{s+1}| - \varepsilon^{1/r^2} n\big) - s|V_{s+1}| - \varepsilon^{\frac{r-1}{r^2}} n -(r-2)\\
> & ~ \Big(\frac{1}{k} - \varepsilon^{1/r^3}\Big) n.
\end{align*}
Hence, there is a vertex $w_{s+1}\in V_{s+1}\setminus L$ such that $\{w_{s+1},u\}$, $\{w_{s+1},v\}$ and 
$\{w_i,w_{s+1}\}$ ($2\leq i\leq s$) are both dense pairs; and the number of ways to choose such $w_{s+1}$ 
is at least $(1/k - \varepsilon^{1/r^3}) n$. 

According to the choice of $u,v,w_2,\ldots,w_{k}$, we know that $\{u,w_{2},\ldots,w_{r}\}$ is not an 
edge of $H$; otherwise we can find an $F_{k+1}^{(r)}$ in $H$. For each $2 \leq s \leq k$, since the 
number of ways to choose $w_{s}$ is at least $(1/k - \varepsilon^{1/r^3}) n$, we find that $u$ is 
contained in at least $(n/k - \varepsilon^{1/r^3} n)^{r-1} > \varepsilon^{1/r^2} n^{r-1}$ missing 
edges, which implies that $u\in M$, a contradiction. This completes the proof of Lemma \ref{lem:independent-set-F}.
\end{proof}

For each $u\in V$, let $M(u)$ denote the set of edges containing $u$ and at least a vertex in $M$. 

\begin{lemma}
For each $u\in V(H)$, $|M(u)| < \varepsilon^{3/(2r^2)} n^{r-1}$.
\end{lemma}

\begin{proof}
Let $u$ be any vertex of $H$. By the definition of $M(u)$ we get  
\begin{align*}
|M(u)| & \leq \sum_{i=1}^{r-1} \binom{|M|}{i} \binom{n-1-i}{r-1-i} \\
& < \sum_{i=1}^{r-1} \frac{|M|^i \cdot n^{r-1-i}}{i! (r-1-i)!} \\
& = \frac{n^{r-1}}{(r-1)!} \Big[\Big(1 + \frac{|M|}{n}\Big)^{r-1} - 1\Big].
\end{align*}
In view of $|M| < \varepsilon^{\frac{r-1}{r^2}} n$ we have 
\[
\Big(1 + \frac{|M|}{n}\Big)^{r-1} - 1 < \Big(1 + \varepsilon^{\frac{r-1}{r^2}}\Big)^{r-1} - 1
< 2(r-1) \varepsilon^{\frac{r-1}{r^2}}.
\]
Combining these inequalities above, gives 
$|M(u)| < \varepsilon^{3/(2r^2)} n^{r-1}$.
\end{proof}

Recall that $\ell = \varepsilon^{3/(2r^2)} n^{r-1}$, we obtain the following assertion.

\begin{lemma}\label{lem:degree-vertex-in-L-F}
For each $u\in M$, $d_H(u) < \binom{k-1}{r-1}\left(\frac{n}{k}\right)^{r-1} - 2\ell$.
\end{lemma}

\begin{proof}
For any $u\in M$, we assume, without loss of generality, that $u\in V_1$. 
We distinguish the following two cases. 

{\bfseries Case 1.} There is a bad edge $e$ in $H$ such that $e\cap L=\{u\}$. 

We show that there is some $j\in \{2,\ldots,k\}$ such that $|DV_j(u)|\leq k\varepsilon^{1/r^2} n$. 
Otherwise, we have $|DV_j(u)| > k\varepsilon^{1/r^2} n$ for each $2\leq j\leq k$. Without 
loss of generality, we assume that $e\cap (V_j\setminus M)=\{w_j\}$ for each $2\leq j\leq r-1$. 
As the same discussion in Lemma \ref{lem:independent-set-F}, we can find $(k-r+1)$ vertices 
$\{w_{r},\ldots,w_k\}$ which, together with vertices in $e$, form the core of $F_{k+1}^{(r)}$, 
a contradiction. 

Using the same arguments as in the proof of Case 1 in Lemma \ref{lem:degree-vertex-in-L}, we arrive 
at $d_H(u) < \binom{k-1}{r-1}\left(\frac{n}{k}\right)^{r-1} - 2\ell$.

{\bfseries Case 2.} Each bad edge $e\in E_u$ in $H$ satisfies $|e\cap L|\geq2$.

In this case, the number of edges in $H$ containing $u$ is at most
\[
\binom{k-1}{r-1} \left(\frac{1}{k} + \varepsilon^{1/r}\right)^{r-1} n^{r-1}
-\varepsilon^{5/4r^2}n^{r-1} + \ell 
< \binom{k-1}{r-1} \left(\frac{n}{k}\right)^{r-1} - 2\ell,
\]
where the inequality follows from $\varepsilon^{3/(2r^2)} \ll \varepsilon^{5/4r^2}$ and 
$\ell = \varepsilon^{3/(2r^2)} n^{r-1}$.

This completes the proof of Lemma \ref{lem:degree-vertex-in-L-F}.
\end{proof}

\begin{lemma}\label{lem:M-emptyset}
$M = \emptyset$.
\end{lemma}

\begin{proof}
Suppose to the contrary that there is a vertex $u\in M$. Let $\bm{x}$ be a nonnegative unit 
eigenvectors of $H$ corresponding to $\lambda^{(p)}(H)$. Let $z$ be a vertex such that 
$x_z = \max_{w\in V(H)}\{x_w\}$ and we may assume without loss of generality that $z\in V_1$. 

We first show that $z\notin M$. By eigenvalue\,--\,eigenvector equation \eqref{eq:eigenequation} 
for $\lambda^{(p)} (H)$ with respect to $z$, we see
\[ 
\lambda^{(p)} (H)\cdot x_z^{p-1} = (r-1)! \sum_{e\in E_z} \prod_{i\in e\setminus\{z\}} x_i 
\leq (r-1)! \cdot d_H(z) x_z^{r-1},
\]
which, together with $x_z \geq n^{-1/p}$ and Lemma \ref{lem:lambdaT}, gives that 
\begin{align*}
d_H(z) & \geq \frac{\lambda^{(p)} (H)}{(r - 1)!} \cdot x_z^{p-r} >
\binom{k-1}{r-1}\left(\frac{n}{k}\right)^{r-1} - 2\ell.
\end{align*}
Thus, by Lemma \ref{lem:degree-vertex-in-L-F}, $z\notin M$.

Now, we shall construct a new $r$-graph on $n$ vertex that is $F_{k+1}^{(r)}$-free but has larger 
$p$-spectral radius, leading to a contradiction and thus completing the proof of this lemma. 
To this end, let us denote $SE$ the set of edges in $H$ containing at least one spare pair. 
We also denote $SE_z := SE \cap E_z$, $E_z^c=E_z\setminus (SE_z\cup L(z))$, 
and $E_z^u=\{(e\setminus\{z\})\cup \{u\}: e\in E_z^c\}$.
Construct an $r$-graph $G$ with $V(G) = V(H)$ and 
\[ 
E(G) = \big( E(H)\setminus E_u\big) \cup E_z^u.
\]

Next, we show that $G$ is $F_{k+1}^{(r)}$-free. Assume for contradiction that $F_{k+1}^{(r)}\subseteq G$. 
This implies that the copy of $F_{k+1}^{(r)}$ in $G$ must contain the vertex $u$. If $u$ is not contained 
in the core of $F_{k+1}^{(r)}$, then according to the construction of $G$, there is a copy of 
$F_{k+1}^{(r)}$ in $G$ that does not contain $u$, a contradiction. Thus, $u$ is contained in the core 
of $F_{k+1}^{(r)}$. It follows that $z$ is not contained in the core. Let $\{u,w_1,w_2,\ldots,w_k\}$ be 
the core of $F_{k+1}^{(r)}$. By the construction of $G$, each pair in $\{u,w_1,w_2,\ldots,w_k\}$ is dense in $G$. 
We claim that one can find a copy of $F_{k+1}^{(r)}$ in $G$ with the core $\{u,w_1,w_2,\ldots,w_k\}$ 
and $z$ is not a vertex in the copy of $F_{k+1}^{(r)}$. If $z$ is a vertex in $F_{k+1}^{(r)}$, then 
there exist $i$ and $j$ ($1\leq i<j\leq k$), and an edge $e$ in the copy of $F_{k+1}^{(r)}$ such that 
$\{z,w_i,w_j\}\subseteq e$. Since $\{w_i,w_j\}$ is a dense pair, we can find a copy of $F_{k+1}^{(r)}$ 
without $z$ whose core is $\{u,w_1,w_2,\ldots,w_k\}$. This implies that $\{z,w_1,w_2,\ldots,w_k\}$ 
forms a core of the copy of $F_{k+1}^{(r)}$ in $H$, a contradiction.

Finally, using the same arguments as in the proof of Lemma \ref{lem:L-emptyset},
we have $\lambda^{(p)} (G) > \lambda^{(p)} (H)$, a contradiction completing the proof of 
Lemma \ref{lem:M-emptyset}.
\end{proof}

\noindent \emph{Proof of Theorem \ref{thm:main-2}}.
By Lemma \ref{lem:independent-set-F} and Lemma \ref{lem:M-emptyset}, we conclude that $H$ is a 
subgraph of $T$. Given the maximality of $\lambda^{(p)} (H)$, it follows that $H\cong T_r(n,k)$ 
from Lemma \ref{lem:k-partite-upper-bound}. \hfill $\Box$

\section{Concluding remarks}

Let $p > 1$ and $\mathcal{F}$ be a family of $r$-graphs. We denote by $\SPEX_p (n, \mathcal{F})$ 
the class of $r$-graphs that attain the maximum $p$-spectral radius among all $\mathcal{F}$-free 
$n$-vertex $r$-graphs. Let $\spex_p (n, \mathcal{F})$ be the $p$-spectral radius of $r$-graphs 
in $\SPEX_p (n, \mathcal{F})$. If $p=r$, we write $\SPEX (n, \mathcal{F})$ and $\spex (n, \mathcal{F})$ 
instead of $\SPEX_p (n, \mathcal{F})$ and $\spex_p (n, \mathcal{F})$, respectively. 

\subsection{Relationship between $\SPEX_p (n, \mathcal{F})$ and $\EX (n, \mathcal{F})$}

In this paper, we show that $\SPEX_p (n, H_{k+1}^{(r)}) = \SPEX_p (n, F_{k+1}^{(r)}) = \{T_r(n,k)\}$  
for $p\geq r$ and sufficiently large $n$. Notably, our proofs avoid relying on the edge extremal results previously 
established by Pikhurko (Corollary \ref{coro:Pikhurko2013}) and Mubayi--Pikhurko (Corollary \ref{coro:Mubayi-Pikhurko2007}). 
Hence, by letting $p\to\infty$, we obtain the corresponding edge extremal results, providing new proofs of 
Mubayi and Pikhurko’s results from spectral viewpoint.


On the other hand, our main results also imply that  
$\SPEX_p (n, H_{k+1}^{(r)}) = \EX (n, H_{k+1}^{(r)})$ and $\SPEX_p (n, F_{k+1}^{(r)}) = \EX (n, F_{k+1}^{(r)})$.
A question that naturally arises from these findings is as follows.

\begin{problem}\label{problem-1}
Let $p>r-1$ and $\mathcal{F}$ be a family $r$-graph. Characterize all $\mathcal{F}$ such that
\[
\SPEX_p (n, \mathcal{F})\subseteq \EX (n, \mathcal{F})
\]
for sufficiently large $n$.
\end{problem}

There are numerous examples for Problem \ref{problem-1} when $r=p=2$. For instance, consider 
$\mathcal{F} = \{K_{t+1}\}$ and $\mathcal{F} = \{C_{2k+1}\}$. For detailed descriptions, we 
refer readers to the recent papers \cite{Wang-Kang-Xue2023}, \cite{Fang-Tait-Zhai2024} 
and \cite{Byrne-Desai-Tait2024}. For $r>2$, only a few examples $\mathcal{F}$ for 
Problem \ref{problem-1} are known, which are listed in Table \ref{table-1}.

\begin{table}
\centering
\begin{tabular}{|c|c|c|}
\hline 
$\mathcal{F}$ & $\SPEX_p (n, \mathcal{F})$ vs. $\EX (n, \mathcal{F})$ & References \\ 
\hline
$H_{k+1}^{(r)}$ & $\SPEX_p (n, \mathcal{F}) = \EX (n, \mathcal{F})$ & current paper \\
\hline 
$\mathcal{K}_{k+1}^{(r)}$ & $\SPEX_p (n, \mathcal{F}) = \EX (n, \mathcal{F})$ & current paper \\
\hline 
Berge-$K_{k+1}$ & $\SPEX_p (n, \mathcal{F}) = \EX (n, \mathcal{F})$ & current paper \\
\hline
$F_{k+1}^{(r)}$ & $\SPEX_p (n, \mathcal{F}) = \EX (n, \mathcal{F})$ & current paper \\
\hline 
$\mathcal{F}_{k+1}^{(r)}$ & $\SPEX_p (n, \mathcal{F}) = \EX (n, \mathcal{F})$ & current paper \\
\hline
$\{F_4, F_5\}$ & $\SPEX_p (n, \mathcal{F}) = \EX (n, \mathcal{F})$ & \cite{Ni-Liu-Kang2024} \\
\hline
Fano plane & $\SPEX_p (n, \mathcal{F}) = \EX (n, \mathcal{F})$ & \cite{Keevash-Lenz-Mubayi2014} \\
\hline
\end{tabular}
\caption{Known results for hypergraph spectral Tur\'an-type problem}
\label{table-1}
\end{table}

As noted by Keevash, Lenz, and Mubayi \cite{Keevash-Lenz-Mubayi2014}, it would be more interesting 
to obtain $p$-spectral results for $r$-graphs where the spectral extremal example differs from the 
usual extremal example.

\begin{problem}\label{problem-2}
Let $p>r-1$ and $\mathcal{F}$ be a family $r$-graph. Characterize all $\mathcal{F}$ such that
\[
\SPEX_p (n, \mathcal{F}) \cap \EX (n, \mathcal{F}) = \emptyset
\]
for sufficiently large $n$.
\end{problem}

For $r=p=2$, consider the family $\mathcal{F}$ pertinent to Problem \ref{problem-2}. A prototypical 
instance of such a family is given by $\mathcal{F} = \{C_4\}$. Nikiforov \cite{Nikiforov2009} and 
Zhai and Wang \cite{Zhai-Wang2012} showed that the only member in $\SPEX (n, C_4)$ is the friendship 
graph ($\lfloor n/2\rfloor$ triangles sharing a single common vertex) if $n$ is odd;
the graph obtained from a star of order $n$ by adding $n/2 - 1$ disjoint additional edges if $n$ 
is even. However, F\"uredi \cite{Furedi1998} proved that if $q$ large enough, the unique graph 
in $\EX (q^2 + q + 1, C_4)$ is a polarity graph of projective plane. It follows that 
$\SPEX (n, C_4) \cap \EX (n, C_4) = \emptyset$ for large $n$ with the form $n = q^2 + q + 1$.
Another example is given by $\mathcal{F} = \{W_{2k+1}\}$, where $W_{2k+1}$ is the wheel graph 
on $2k+1$ vertices formed by joining a vertex to all of the vertices in a cycle on $2k$ vertices.
Cioab\u{a}, Desai, and Tait \cite{Cioaba-Desai-Tait2022} proved that 
$\SPEX (n, W_{2k+1}) \cap \EX (n, W_{2k+1}) = \emptyset$ when $k = 7$ or $k = 9$
and $n$ is sufficiently large. It would be interesting to obtain examples $\mathcal{F}$ 
pertinent to Problem \ref{problem-2} for $r>2$.
\par\vspace{2mm}

It is worth noting that the set $\SPEX (n, \mathcal{F})$ contains a unique member for $\mathcal{F}$ 
listed in Table \ref{table-1}. This naturally raises the question: It is possible to have more than 
one member in $\SPEX (n, \mathcal{F})$ for some $\mathcal{F}$? In the following subsection we provide 
such an example for $r > 2$.


\subsection{A spectral Tur\'an-type problem with multiple extremal hypergraphs}

A $3$-graph $H$ is called \emph{semibipartite}\footnote{In spectral hypergraph theory, it is often 
called hm-bipartite (see, e.g. \cite{Hu-Qi2014}).} if there exists a partition $V(H) = A\cup B$ 
such that $|e\cap A| = 1$ and $|e\cap B| = 2$ for all $e\in E(H)$. Define $G_n^1$ as the edge 
maximal semibipartite $3$-graph on $n$ vertices with vertex partition $V(G_n^1) = A\cup B$ 
where $|A| = \lfloor n/3\rfloor$ and $\lceil 2n/3\rceil$. 
Let $G_6^2$ be the $3$-graph with vertex set $[6]$, whose complement is  
\[
\overline{G_6^2} = \{123, 126, 345, 456\}.
\]
For $n>6$, let $G_n^2$ be a $3$-graph on $n$ vertices that is a blow-up of $G_6^2$ with the 
maximum number of edges\footnote{Simple calculations reveal that each part in $G_n^2$ has 
size either $\lfloor n/6\rfloor$ or $\lceil n/6\rceil$.}. A $3$-graph $H$ is 
called \emph{$G_6^2$-colorable} if it is a subgraph of a blow-up of $G_6^2$. 

Liu and Mubayi \cite{Liu-Mubayi2022} construct a finite family of triple systems 
$\mathcal{M}$, and determine its Tur\'an number. The family $\mathcal{M}$ is the 
union of the following three finite families $M_1$, $M_2$ and $M_3$, where
\begin{enumerate}\setlength{\itemsep}{0pt}
\item[$(1)$] $M_1$ is the set containing the complete $3$-graph on five vertices with one edge removed.

\item[$(2)$] $M_2$ is the collection of all $3$-graphs in $\mathcal{K}_7^{(3)}$
 with a core whose induced subgraph has transversal number at least two.

\item[$(3)$] $M_3$ is the collection of all $3$-graphs $F\in\mathcal{K}_6^3$
such that both $F \not\subset G_n^1$ and $F\not\subset G_n^2$
for all positive integers $n$.
\end{enumerate}

Liu and Mubayi \cite{Liu-Mubayi2022} proved that 
\begin{equation}\label{eq:Turan-number-M}
\ex (n, \mathcal{M}) \leq \frac{2n^3}{27}
\end{equation}
holds for all positive integers $n$. Note that both $G_n^1$ and $G_n^2$ are $\mathcal{M}$-free, 
and $e(G_n^1) = e(G_n^2) = 2n^3/27$ when $n$ is a multiple of six, we have 
$\{G_n^1, G_n^2\}\subseteq \EX (n, \mathcal{M})$ if $n$ is a multiple of six.

\begin{proposition}[\cite{Liu-Mubayi2022}]\label{prop:iff-hom-free}
A $3$-graph $H$ is $\mathcal{M}$-free if and only if it is $\mathcal{M}$-hom-free.
\end{proposition}  

The next theorem implies that there are at least two members in $\SPEX_p (n, \mathcal{M})$ 
when $n$ is a multiple of six.

\begin{theorem}
Let $p\geq 1$. Then $\spex (n, \mathcal{M}) \leq 4 n^{3(1-1/p)}/9$.
Moreover, $\{G_n^1, G_n^2\} \subseteq \SPEX_p (n, \mathcal{M})$ if $n$ is a multiple of six.
\end{theorem}

\begin{proof}
Let $H$ be a $\mathcal{M}$-free $3$-graph on $n$ vertices, and let $\bm{x}\in\mathbb{S}_{p,+}^{n-1}$ 
be an eigenvector corresponding to $\lambda^{(p)} (H)$. By Power Mean inequality,
\begin{align}
\lambda^{(p)} (H) & = 6 \sum_{\{i,j,k\}\in E(H)} x_ix_jx_k \nonumber \\
& \leq 6 \big(e(H)\big)^{1-1/p} \Bigg(\sum_{\{i,j,k\}\in E(H)} (x_ix_jx_k)^p\Bigg)^{1/p} \nonumber \\
& \leq \big(6 e(H)\big)^{1-1/p} \cdot (\pi_{\lambda} (\mathcal{M}))^{1/p}. \label{eq:sec6-temporary} 
\end{align}
By Proposition \ref{prop:iff-hom-free} and \eqref{eq:Hom-free-density}, $\pi_{\lambda} (\mathcal{M}) \leq \pi(\mathcal{M}) = 4/9$.
Combining this with \eqref{eq:sec6-temporary} gives $\lambda^{(p)} (H) \leq 4 n^{3(1-1/p)}/9$.

If $n$ is a multiple of six, then by the definition of the $p$-spectral radius
and the inequality \eqref{eq:sec6-temporary} for $G_n^1$ and $G_n^2$, we have 
$\lambda^{(p)} (G_n^1) = \lambda^{(p)} (G_n^2) = 4n^{3(1-1/p)}/9$.
It follows that $\{G_n^1, G_n^2\} \subseteq \SPEX_p (n, \mathcal{M})$.
\end{proof}

As highlighted in Section \ref{sec1}, stability methods play a crucial role in the study of 
hypergraph Tur\'an-type problems. Stability results provide a deep understanding of not just 
the extremal number of edges in a hypergraph avoiding certain forbidden configurations, but 
also the structural properties of hypergraphs that are near this extremal bound. However, 
there are Tur\'an-type problems for hypergraphs that do not have the stability property.
Liu and Mubayi \cite{Liu-Mubayi2022} provide the first such example, which is described 
in the next subsection. 

\subsection{A hypergraph spectral Tur\'an-type problem with no spectral stability}

Liu and Mubayi \cite{Liu-Mubayi2022} proved that there are two near-extremal $\mathcal{M}$-free 
$3$-graphs that transforming one to the other requires us to delete and add $\Omega(n^3)$ edges. 

\begin{theorem}[\cite{Liu-Mubayi2022}]\label{thm:no-stability}
For every $\varepsilon > 0$, there exists $\delta > 0$ and $n_0$ such that the following 
holds for all $n \geq n_0$: Every $\mathcal{M}$-free $3$-graph on $n$ vertices with at least 
$(1 - \delta) \cdot 2n^3/27$ edges can be transformed to a $3$-graph that is either semibipartite 
or $G_6^2$-colorable by removing at most $\varepsilon n$ vertices.    
\end{theorem}

Applying Theorem \ref{thm:no-stability} in conjunction with the proof techniques in 
Lemma \ref{lem:general-criterion}, we conclude that $\mathcal{M}$ does not have the spectral stability.

\begin{theorem}\label{thm:no-spectral-stability}
For every $\varepsilon > 0$, there exists $\delta > 0$ and $n_0$ such that the following holds 
for all $n \geq n_0$: Every $\mathcal{M}$-free $3$-graph $H$ on $n$ vertices with 
$\lambda^{(p)} (H) > (1 - \delta) \cdot 4 n^{3(1-1/p)}/9$ can be transformed to a $3$-graph
that is either semibipartite or $G_6^2$-colorable by removing at most $\varepsilon n$ vertices.    
\end{theorem}

\begin{proof}
Let $\varepsilon > 0$ be given. By Theorem \ref{thm:no-stability}, there exists a positive 
$\delta_1$ such that any $n$-vertex $\mathcal{M}$-free $3$-graph with at least $(1 - \delta_1) \cdot 2n^3/27$ 
edges can be transformed to a $3$-graph that is either semibipartite or $G_6^2$-colorable by 
removing at most $\varepsilon n$ vertices.    

Let $H$ be an $n$-vertex $\mathcal{M}$-free $3$-graph. It follows from \eqref{eq:sec6-temporary} that
\[ 
\lambda^{(p)} (H) \leq (6\cdot e(H))^{1-1/p} \cdot \pi_{\lambda}(\mathcal{M})^{1/p},
\]
which, together with $\pi_{\lambda}(\mathcal{M}) \leq 4/9$, implies that 
\[
\lambda^{(p)} (H) \leq 6\cdot \Big(\frac{2}{27}\Big)^{1/p} \cdot e(H)^{1-1/p}.
\]
On the other hand, choose $\delta = (1-1/p) \delta_1$ and assume that 
$\lambda^{(p)} (H) > (1 - \delta) \cdot 4 n^{3(1-1/p)}/9$. We have
\[
6\cdot \Big(\frac{2}{27}\Big)^{1/p} \cdot e(H)^{1-1/p} > 
(1 - \delta) \cdot \frac{4}{9} n^{3(1-1/p)},
\]
which follows that 
\[
e(H) > (1 - \delta)^{p/(p-1)} \cdot \frac{2}{27} \cdot n^3 
> (1 - \delta_1) \cdot \frac{2n^3}{27}.
\]
Finally, by Theorem \ref{thm:no-stability} we get the desired result.
\end{proof}

\section*{Acknowledgments}

The research of Zhenyu Ni was partially supported by the National Nature Science 
Foundation of China (No. 12201161), and Hainan Provincial Natural Science Foundation 
(No. 122QN218). Jing Wang was partially supported by the National Nature Science 
Foundation of China (No. 12301437), and the China Postdoctoral Science Foundation 
(No. 2023M731013). Liying Kang was partially supported by the National Nature 
Science Foundation of China (No. 12331012).

\end{document}